\documentclass[11pt,a4paper]{article}

\usepackage{epsf,epsfig,amsfonts,amsgen,amsmath,amstext,amsbsy,amsopn,amsthm,cases,listings,color
}
\usepackage{ebezier,eepic}
\usepackage{color}
\usepackage{multirow}
\usepackage{epstopdf}
\usepackage{graphicx}
\usepackage{pgf,tikz}
\usepackage{mathrsfs}
\usepackage[marginal]{footmisc}
\usepackage{enumitem}
\usepackage[titletoc]{appendix}
\usepackage{booktabs}
\usepackage{url}
\usepackage{mathtools}

\usepackage{pgfplots}
\usepackage{authblk}
\usepackage{amssymb}

\usepackage{wasysym}

\usepackage{empheq}

\usepackage{dsfont}

\usepackage{subfigure}
\usepackage{mathrsfs}
\usepackage{wasysym} 
\usetikzlibrary{arrows}
\usepackage{aligned-overset}

 \usepackage[backref=page]{hyperref}

\allowdisplaybreaks[1]

\definecolor{uuuuuu}{rgb}{0.27,0.27,0.27}
\definecolor{sqsqsq}{rgb}{0.1255,0.1255,0.1255}

\newtheorem{definition}{Definition} [section]
\newtheorem{theorem}[definition]{Theorem}
\newtheorem{lemma}[definition]{Lemma}
\newtheorem{proposition}[definition]{Proposition}

\newtheorem{corollary}[definition]{Corollary}
\newtheorem{conjecture}[definition]{Conjecture}
\newtheorem{claim}[definition]{Claim}
\newtheorem{problem}[definition]{Problem}

\newtheorem{fact}[definition]{Fact}
\newenvironment{pf}{\noindent{\bf Proof}}

\newcommand{\uproduct}{\mathbin{\;{\rotatebox{90}{\textnormal{$\small\Bowtie$}}}}}

\newcommand{\blow}[2]{#1(\!(#2)\!)}

\begin{document}
\title{\bf\Large A step towards a general density Corr\'{a}di--Hajnal Theorem}
\date{\today}
\author[1]{Jianfeng Hou\thanks{Research was supported by National Natural Science Foundation of China (Grant No. 12071077). Email: \texttt{jfhou@fzu.edu.cn}}}
\author[1]{Heng Li\thanks{Email: \texttt{hengli.fzu@gmail.com}}}
\author[2]{Xizhi Liu\thanks{Research was supported by ERC Advanced Grant 101020255. Email: \texttt{xizhi.liu@warwick.ac.uk}}}
\author[3]{Long-Tu Yuan\thanks{Research was supported by the National Natural Science Foundation of China (No. 12271169) and Science and Technology Commission of Shanghai Municipality, China (No. 22DZ2229014). \newline
Email: \texttt{ltyuan@math.ecnu.edu.cn}}}
\author[1]{Yixiao Zhang\thanks{Email: \texttt{fzuzyx@gmail.com}}}
\affil[1]{Center for Discrete Mathematics,
            Fuzhou University, Fujian, 350003, China}
\affil[2]{Mathematics Institute and DIMAP,
            University of Warwick, 
            Coventry, CV4 7AL, UK}
\affil[3]{School of Mathematical Sciences and Shanghai Key Laboratory of PMMP,
            \newline
            East China Normal University, 
            Shanghai, 200240, China}
\maketitle
\begin{abstract}
%
For a nondegenerate $r$-graph $F$, large $n$, and $t$ in the regime $[0, c_{F} n]$, where $c_F>0$ is a constant depending only on $F$, 
we present a general approach for determining the maximum number of edges in an $n$-vertex $r$-graph that does not contain $t+1$ vertex-disjoint copies of $F$. 
In fact, our method results in a rainbow version of the above result and includes a  characterization of the extremal constructions. 

Our approach applies to many well-studied hypergraphs (including graphs) such as 
the edge-critical graphs, the Fano plane,  the generalized triangles, hypergraph expansions, the expanded triangles, and hypergraph books. 
Our results extend old results of Simonovits~\cite{SI68} and Moon~\cite{Moon68} on complete graphs and can be viewed as a step towards a general density version of the classical Corr\'{a}di--Hajnal Theorem~\cite{CH63}.

\medskip

\textbf{Keywords:} Hypergraph Tur\'{a}n problems, the Corr\'{a}di--Hajnal Theorem, $F$-matching, stability, vertex-extendability. 
\end{abstract}
\section{Introduction}\label{SEC:Introduction}
\subsection{Motivation}\label{SUBSEC:Motivation}
Fix an integer $r\ge 2$, an $r$-graph $\mathcal{H}$ is a collection of $r$-subsets of some finite set $V$. We identify a hypergraph $\mathcal{H}$ with its edge set and use $V(\mathcal{H})$ to denote its vertex set. The size of $V(\mathcal{H})$ is denoted by $v(\mathcal{H})$. 

Given two $r$-graphs $F$ and $\mathcal{H}$ we use $\nu(F, \mathcal{H})$ to denote the maximum of $k\in \mathbb{N}$ such that there exist $k$ vertex-disjoint copies of $F$ in $\mathcal{H}$. We call $\nu(F, \mathcal{H})$ the \textbf{$F$-matching number} of $\mathcal{H}$.
If $F = K_{r}^{r}$ (i.e. an edge), then 
we use $\nu(\mathcal{H})$ to represent $\nu(F, \mathcal{H})$ for simplicity. 
The number $\nu(\mathcal{H})$ is also known as the \textbf{matching number} of $\mathcal{H}$. 

The study of the following problem  encompasses several central topics in Extremal Combinatorics.
Given an $r$-graph $F$ and integers $n, t\in \mathbb{N}$:
\begin{align*}
    \textit{What kinds of constraints on an $n$-vertex $r$-graph $\mathcal{H}$ force it to satisfy $\nu(F, \mathcal{H}) \ge t+1$?}
\end{align*}

For $r=2$ and $F = K_{2}$, 
the celebrated Erd\H{o}s--Gallai Theorem~\cite{EG59} states that for
all integers $n, \ell \in \mathbb{N}$ with $t+1 \le n/2$ and for every $n$-vertex graph $G$,
\begin{align*}
|G| > \max \left\{\binom{2t+1}{2}, \binom{n}{2}-\binom{n-t}{2}\right\}
\quad\Rightarrow\quad \nu(G) \ge t+1.      
\end{align*}
Here we use the symbol $\Rightarrow$ to indicate that the constraint on the left side forces the conclusion on the right side. 

Extending the Erd\H{o}s--Gallai Theorem to $r$-graphs for $r\ge 3$ is a major open problem, and the following conjecture of Erd\H{o}s is still open in general
(see e.g. \cite{Frankl13,Frankl17A,Frankl17B, HLS12} for some recent progress on this topic). 

\begin{conjecture}[Erd\H{o}s~\cite{Erdos65}]\label{CONJ:Erdos-matching}
Suppose that $n, t, r \in \mathbb{N}$ satisfy $r\ge 3$ and $t+1 \le n/r$. 
Then for every $n$-vertex $r$-graph $\mathcal{H}$,
\begin{align*}
    |\mathcal{H}|> \max\left\{\binom{r(t+1)-1}{r}, \binom{n}{r}-\binom{n-t}{r}\right\}
    \quad\Rightarrow\quad \nu(\mathcal{H}) \ge t+1. 
\end{align*}
\end{conjecture}

For general $r$-graphs $F$, determining the minimum number of edges in an $n$-vertex $r$-graph $\mathcal{H}$ that guarantees $\nu(F,\mathcal{H}) \ge 1$ is closely related to the Tur\'{a}n problem. 
For our purpose in this work, let us introduce the following notions. 

Fix an $r$-graph $F$, we say another $r$-graph $\mathcal{H}$ is \textbf{$F$-free} if $\nu(F,\mathcal{H}) = 0$. In other words, $\mathcal{H}$ does not contains $F$ as a subgraph. 
%
%
The \textbf{Tur\'{a}n number} $\mathrm{ex}(n,F)$ of $F$ is the maximum number of edges in an $F$-free $r$-graph on $n$ vertices.
The \textbf{Tur\'{a}n density} of $F$ is defined as $\pi(F) := \lim_{n\to \infty}\mathrm{ex}(n,F)/\binom{n}{r}$,
the existence of the limit follows from a simple averaging argument of Katona, Nemetz, and Simonovits~\cite{KNS64} (see Proposition~\ref{PROP:KNS}).

An $r$-graph $F$ is called \textbf{nondegenerate} if $\pi(F) > 0$.
We use $\mathrm{EX}(n,F)$ to denote the collection of all $n$-vertex $F$-free $r$-graphs with exactly $\mathrm{ex}(n,F)$ edges, and
call members in $\mathrm{EX}(n,F)$ the \textbf{extremal constructions} of $F$. 
The study of $\mathrm{ex}(n,F)$ and $\mathrm{EX}(n,F)$ is a central topic in Extremal Combinatorics. 

Much is known when $r=2$, 
and one of the earliest results in this regard is Mantel's theorem~\cite{Mantel07}, which states that $\mathrm{ex}(n,K_3) = \lfloor n^2/4 \rfloor$. 
For every integer $\ell\ge 2$ let $T(n,\ell)$ denote the balanced complete $\ell$-partite graph on $n$ vertices. Here, balanced means that the sizes of any two parts differ by at most one.
We call $T(n,\ell)$ the \textbf{Tur\'{a}n graph}, and
use $t(n,\ell)$ to denote the number of edges in $T(n,\ell)$. 
The seminal Tur\'{a}n Theorem states that $\mathrm{EX}(n, K_{\ell+1}) = \{T(n,\ell)\}$ for all integers $n \ge \ell \ge 2$. 
%
Later, Tur\'{a}n's theorem was extended to general graphs $F$ in 
the celebrated Erd\H{o}s--Stone--Simonovits Theorem~\cite{ES66,ES46}, which says that
$\pi(F) = \left(\chi(F)-2\right)/\left(\chi(F)-1\right)$.
Here $\chi(F)$ is the \textbf{chromatic number} of~$F$.

For $r\ge 3$, determining $\mathrm{ex}(n, F)$ or even $\pi(F)$ for an $r$-graph $F$ is known to be notoriously hard in general.  
The problem of determining $\pi(K_{\ell}^{r})$ raised by Tur\'{a}n~\cite{TU41}, 
where $K_{\ell}^{r}$ is the complete $r$-graph on $\ell$ vertices, is still wide open for all $\ell>r\ge 3$. 
Erd\H{o}s offered $\$ 500$ for the determination of any $\pi(K_{\ell}^{r})$ with $\ell > r \ge 3$ and $\$ 1000$ for all $\pi(K_{\ell}^{r})$ with $\ell > r \ge 3$.
We refer the reader to an excellent survey~\cite{KE11} by Keevash for related results before 2011. 

Another related central topic in Extremal Combinatorics is the Factor Problem. 
We say an $r$-graph $\mathcal{H}$ has an \textbf{$F$-factor} if it contains a collection of vertex-disjoint copies of $F$ that covers all vertices in $V(\mathcal{H})$. 
In other words, $\nu(F, \mathcal{H}) = \frac{v(\mathcal{H})}{v(F)}$ (in particular, $v(F) \mid v(\mathcal{H})$). 

For an $r$-graph $\mathcal{H}$ and a vertex $v\in V(\mathcal{H})$
the \textbf{degree} $d_{\mathcal{H}}(v)$ of $v$ in $\mathcal{H}$ is the number of edges in $\mathcal{H}$ containing $v$.
We use $\delta(\mathcal{H})$, $\Delta(\mathcal{H})$, and $d(\mathcal{H})$ to denote the \textbf{minimum degree}, the \textbf{maximum degree}, and the \textbf{average degree} of $\mathcal{H}$, respectively.
We will omit the subscript $\mathcal{H}$ if it is clear from the context. 

A classical theorem of Corr\'{a}di and Hajnal~\cite{CH63} implies the following result for $K_3$. 

\begin{theorem}[Corr\'{a}di--Hajnal~\cite{CH63}]\label{THM:CH-min-deg}
    Suppose that $n, t \in \mathbb{N}$ are integers with $t\le n/3$. 
    Then for every $n$-vertex graph $G$,
    \begin{align*}
        \delta(G) \ge t + \left\lfloor \frac{n-t}{2} \right\rfloor 
        \quad\Rightarrow\quad
        \nu(K_3, G) \ge t. 
    \end{align*}
    In particular, if $3 \mid n$, then every $n$-vertex graph $G$ with $\delta(G) \ge 2n/3$ contains a $K_3$-factor. 
\end{theorem}

Later, 
Theorem~\ref{THM:CH-min-deg} was extended to all complete graphs in 
the classical Hajnal--Szemer\'{e}di Theorem~\cite{HS70}, which implies that 
for all integers $n\ge \ell \ge 2$, $t \le \lfloor n/(\ell+1) \rfloor$, and for every $n$-vertex graph $G$,
\begin{align*}
    \delta(G) \ge  t + \left\lfloor \frac{\ell-1}{\ell}(n-t) \right\rfloor 
        \quad\Rightarrow\quad
        \nu(K_{\ell+1}, G) \ge t. 
\end{align*}
%

%
For further related results, we refer the reader to a survey~\cite{KO09} by K\"{u}hn and Osthus. 

In this work, we are interested in density constraints that force an $r$-graph to have large $F$-matching number, where $F$ is a nondegenerate $r$-graph. 
Since our results are closely related to the Tur\'{a}n problem of $F$, 
we abuse the use of notation by letting $\mathrm{ex}\left(n, (t+1)F\right)$ denote the maximum number of edges in an $n$-vertex $r$-graph $\mathcal{H}$ with $\nu(F, \mathcal{H}) < t+1$. 

Given two $r$-graphs $\mathcal{G}$ and $\mathcal{H}$ whose vertex sets are disjoint, 
we define the \textbf{join} $\mathcal{G} \uproduct \mathcal{H}$ of $\mathcal{G}$ and $\mathcal{H}$ to be the $r$-graph obtained from $\mathcal{G}\sqcup \mathcal{H}$ (the vertex-disjoint union of $\mathcal{G}$ and $\mathcal{H}$) by adding all $r$-sets that have nonempty intersection with both $V(\mathcal{G})$ and $V(\mathcal{H})$. 
For simplicity, we define 
the join of an $r$-graph $\mathcal{H}$ and a family $\mathcal{F}$ of $r$-graphs as 
$\mathcal{H} \uproduct \mathcal{F} := \left\{\mathcal{H} \uproduct \mathcal{G} \colon \mathcal{G} \in \mathcal{F}\right\}$. 

Erd\H{o}s~\cite{Erdos62} considered the density problem for $K_3$ and proved the following result. 

\begin{theorem}[Erd\H{o}s~\cite{Erdos62}]\label{THM:Erdos-disjoint-triangle}
    Suppose that $n, t\in \mathbb{N}$ and $t\le \sqrt{n/400}$. 
    Then 
    \begin{align*}
        \mathrm{EX}\left(n, (t+1)K_3\right) = \{K_{t} \uproduct T(n-t,2)\}. 
    \end{align*}
\end{theorem}

Later, Moon~\cite{Moon68} extended it to all complete graphs. 

\begin{theorem}[Moon~\cite{Moon68}]\label{THM:Moon-disjoint-clique}
    Suppose that integers $n, t, \ell\in \mathbb{N}$ satisfy $\ell \ge 2$, $t\le \frac{2n-3\ell^2+2\ell}{\ell^3+2\ell^2+\ell+1}$, and $\ell \mid (n-t)$. 
    Then 
    \begin{align}\label{equ:Moon}
        \mathrm{EX}\left(n, (t+1)K_{\ell+1}\right) = \left\{K_{t}\uproduct T(n-t,\ell)\right\}.  
    \end{align}
\end{theorem}

It is worth mentioning that, in fact, for $\ell = 2$, Moon proved that the constraint $\ell \mid (n-t)$ can be removed, and moreover, (\ref{equ:Moon}) holds for all $t \le \frac{2n-8}{9}$.
For $\ell \ge 3$, Moon remarked in~\cite{Moon68} that there are some difficulties to remove the constraint  $\ell \mid (n-t)$. 
Nevertheless, the divisibility constraint is not required in our results.
Meanwhile, Simonovits~\cite{SI68} also considered this problem and proved that
if $t\ge 1$ and $\ell \ge 2$ are fixed integers, then (\ref{equ:Moon}) holds for all sufficiently large $n$.

It becomes much more complicated when extending Theorem~\ref{THM:Moon-disjoint-clique} to larger $t$.
Indeed, a full density version of the Corr\'{a}di--Hajnal Theorem was obtained only very recently by Allen, B\"{o}ttcher, Hladk\'{y}, and Piguet~\cite{ABHP15} for large $n$. 
Their results show that, interestingly, there are four different extremal constructions for four different regimes of $t$, and the construction $K_{t} \uproduct T(n-t, 2)$ is extremal only for $t \le \frac{2n-6}{9}$. For the other three extremal constructions, we refer the reader to their paper for details.
For larger complete graphs, it seems that there are even no conjectures for the extremal constructions in general (see remarks in the last section of~\cite{ABHP15}).  

The objective of this work is to provide a general approach to determine $\mathrm{ex}(n, (t+1)F)$ for nondegenerate hypergraphs (including graphs) $F$ when $n$ is sufficiently large and $t$ is within the range of $[0, c_F n]$, where $c_F>0$ is a small constant depending only on $F$. 
Our main results are stated in the next section after the introduction of some necessary definitions. 
We hope our results could shed some light on a full generalization of the density version of the  Corr\'{a}di--Hajnal Theorem. 

\subsection{Main results}\label{SUBSEC:main-result}
Given an $r$-graph $F$ and an integer $n \in \mathbb{N}$ define
\begin{align*}
    \delta(n, F) 
         := \mathrm{ex}(n, F) - \mathrm{ex}(n-1, F) \quad \text{and} \quad
    d(n, F) 
         :=  \frac{r\cdot \mathrm{ex}(n, F)}{n}. 
\end{align*}
Observe that $d(n,F)$ is the average degree of hypergraphs in $\mathrm{EX}(n,F)$, and $\delta(n,F)$ is a lower bound for the minimum degree of hypergraphs in $\mathrm{EX}(n,F)$ (see Fact~\ref{FACT:min-degree-extremal}). 

The following two definitions are crucial for our main results. 
The first definition concerns the maximum degree of a near-extremal $F$-free $r$-graph. 

\begin{definition}[\textbf{Boundedness}]\label{DEFINITION:Bad-Degree}
Let $f_1, f_2 \colon \mathbb{N} \to \mathbb{R}$ be two nonnegative functions.
An $r$-graph $F$ is $\left(f_1, f_2\right)$-bounded if 
every $F$-free $r$-graph $\mathcal{H}$ on $n$ vertices with 
average degree at least $d(n, F) - f_1(n)$ 
satisfies $\Delta(\mathcal{H}) \le d(n, F) + f_2(n)$, i.e. 
\begin{align*}
    d(\mathcal{H}) \ge d(n, F) - f_1(n) \quad\Rightarrow\quad \Delta(\mathcal{H}) \le d(n, F) + f_2(n). 
\end{align*}
\end{definition}

Later we will prove that families with certain stability properties also have good boundedness (see Theorem~\ref{THM:vertex-extend-bounded}). 

The next definition concerns the smoothness of the Tur\'{a}n function $\mathrm{ex}(n,F)$. 

%
\begin{definition}[\textbf{Smoothness}]\label{DEFINITION:smoothness}
Let $g \colon \mathbb{N} \to \mathbb{R}$ be a nonnegative function. 
The Tur\'{a}n function $\mathrm{ex}(n, F)$ of an $r$-graph $F$ is $g$-smooth if 
\begin{align*}
    \left|\delta(n, F) - d(n-1, F)\right| \le g(n)  
    \quad\text{holds for all } n \in \mathbb{N}. 
\end{align*}
\end{definition}

Assumptions on the smoothness of $\mathrm{ex}(n, F)$ were used by several researchers before. 
See e.g. \cite{AKSV14,JLM21}  for degenerate graphs 
and see e.g.~\cite[Theorem~1.4]{KLM14} for nondegenerate hypergraphs. 

Now we are ready to state our main result. 

\begin{theorem}\label{THM:Main-simple}
Fix integers $m \ge r \ge 2$ and a nondegenerate $r$-graph $F$ on $m$ vertices. 
Suppose that there exists a constant $c >0$ such that for all sufficiently large $n\in \mathbb{N} \colon$
\begin{enumerate}[label=(\alph*)]
\item
    $F$ is $\left(c\binom{n}{r-1}, \frac{1-\pi(F)}{4m}\binom{n}{r-1}\right)$-bounded, and 
\item
    $\mathrm{ex}(n, F)$ is $\frac{1-\pi(F)}{8m}\binom{n}{r-1}$-smooth. 
\end{enumerate}
Then there exists $N_0$ such that for all integers $n \ge N_0$ and $t \le \min\left\{\frac{c}{4erm}n, \frac{1-\pi(F)}{64rm^2}n\right\}$, 
we have 
\begin{align}\label{equ:turan-number-structure-main}
    \mathrm{EX}\left(n, (t+1)F\right)
        = K_{t}^r \uproduct \mathrm{EX}(n-t, F),
\end{align}
and, in particular, 
\begin{align}\label{equ:turan-number-main}
    \mathrm{ex}\left(n, (t+1)F\right) 
        = \binom{n}{r} - \binom{n-t}{r} + \mathrm{ex}(n-t, F).
\end{align}
\end{theorem}
\textbf{Remark.}
Note that one cannot hope that (\ref{equ:turan-number-main}) holds for all nondegenerate $r$-graphs. 
Indeed, if we let $F = 2K_3$ and let $t \ge 2$, then
\begin{align*}
    \mathrm{ex}(n,(t+1)F) = \mathrm{ex}(n,(2t+2)K_3)
    & \ge \binom{n}{2}-\binom{n-2t-1}{2}+ \left\lfloor \frac{(n-2t-1)^2}{4} \right\rfloor \\
    & > \binom{n}{2}-\binom{n-t}{2} + \left\lfloor \frac{(n-1)^2}{4} \right\rfloor + n-1 \\
    & = \binom{n}{2}-\binom{n-t}{2} + \mathrm{ex}(n-t, F).   
\end{align*}

Fix an $r$-graph $F$ on $m$ vertices. 
We say a collection $\left\{\mathcal{H}_1, \ldots, \mathcal{H}_{t+1}\right\}$ of $r$-graphs on the same vertex set $V$ has a \textbf{rainbow $F$-matching} 
if there exists a collection $\left\{S_i \colon i \in [t+1]\right\}$ of pairwise disjoint $m$-subsets of $V$ such that $F \subset \mathcal{H}_{i}[S_i]$ for all $i\in [t+1]$.

Recently, there has been considerable interest in extending some classical results to a rainbow version. See e.g. \cite{AH,GLMY22,HLS12,KK21,LWY22,LWY23} for some recent progress on the rainbow version of the Erd\H{o}s Matching Conjecture. 
Here we include the following rainbow version of Theorem~\ref{THM:Main-simple}. 

\begin{theorem}\label{THM:main-rainbow}
The following holds under the assumption of Theorem~\ref{THM:Main-simple}. 
If a collection $\left\{\mathcal{H}_1, \ldots, \mathcal{H}_{t+1}\right\}$ of $n$-vertex $r$-graphs on the same vertex set satisfies
    \begin{align*}
        |\mathcal{H}_i| > \binom{n}{r}-\binom{n-t}{r}+\mathrm{ex}(n-t,F) 
        \quad\text{for all } i\in [t+1], 
    \end{align*}
then $\left\{\mathcal{H}_1, \ldots, \mathcal{H}_{t+1}\right\}$ contains a rainbow $F$-matching. 
\end{theorem}

Observe that (\ref{equ:turan-number-main}) follows immediately
by letting $\mathcal{H}_1 = \cdots = \mathcal{H}_{t+1}$ in Theorem~\ref{THM:main-rainbow}. 
In fact, we will prove Theorem~\ref{THM:main-rainbow} first (which yields (\ref{equ:turan-number-main})), and then we prove (\ref{equ:turan-number-structure-main}) by adding some further argument. 

\subsection{Boundedness and smoothness}\label{SUBSEC:Bound-Smooth}
In this subsection, we present some simple sufficient conditions for an $r$-graph to have good boundedness and smoothness. Before stating our results, let us introduce some necessary definitions. 

For many nondegenerate Tur\'{a}n problems the extremal constructions usually have simple structures. We use the following notions to encode the structural information of a hypergraph. 

Let an \textbf{$r$-multiset} mean an unordered collection of $r$ elements with repetitions allowed. Let $E$ be a collection of $r$-multisets on $[k]$.  
Let $V_1,\dots,V_k$ be disjoint sets and let $V:=V_1\cup\dots\cup V_k$. 
The \textbf{profile} of an $r$-set $X\subseteq V$ (with respect to $V_1,\dots,V_k$) is
the $r$-multiset on $[k]$ that contains $i\in [k]$ with multiplicity $|X\cap V_i|$. 
For an $r$-multiset $Y\subseteq [k]$, 
let $\blow{Y}{V_1,\dots,V_k}$ consist of all $r$-subsets of $V$ whose profile is $Y$. 
The $r$-graph $\blow{Y}{V_1,\dots,V_k}$ is called the \textbf{blowup} of $Y$ (with respect to $V_1,\dots,V_k$) and
the $r$-graph
\begin{align*}
    \blow{E}{V_1,\dots,V_k}:=\bigcup_{Y\in E} \blow{Y}{V_1,\dots,V_k}    
\end{align*}
is called the \textbf{blowup} of $E$ (with respect to $V_1,\dots,V_k$).

An \textbf{($r$-uniform) pattern} is a pair $P=(k,E)$ where $k$ is a positive integer and 
$E$ is a collection of $r$-multisets on $[k]$. 
It is clear that pattern is a generalization of $r$-graphs, 
since an $r$-graph is a pattern in which $E$ consists of only simple $r$-sets. 
If it is clear from the context, we will use $E$ to represent the pattern $P$ for simplicity (like what we did for hypergraphs). 
Moreover, if $E$ consists of a single element, we will use this element to represent $E$. 

We say an $r$-graph $\mathcal{G}$ is a \textbf{$P$-construction} on a set $V$ 
if there exists a partition $V = V_1 \cup \cdots \cup V_{k}$ such that $\mathcal{G} = \blow{E}{V_1,\dots,V_k}$. 
An $r$-graph $\mathcal{H}$ is a \textbf{$P$-subconstruction} if it is a subgraph of some $P$-construction. 
For example, the Tur\'{a}n graph $T(n,\ell)$ is a $K_{\ell}$-constrction on $[n]$, and an $\ell$-partite graph is a $K_{\ell}$-subconstrction. 

Let $\Lambda(P, n)$ denote the maximum number of edges in a $P$-construction with $n$ vertices and define the \textbf{Lagrangian} of $P$ as the limit
\begin{align*}
    \lambda(P) := \lim_{n\to \infty} \frac{\Lambda(P, n)}{\binom{n}{r}}. 
\end{align*}
Using a simple averaging argument, one can show that ${\Lambda(P, n)}/{\binom{n}{r}}$ is nonincreasing, and hence, the limit exists. 
We say a pattern $P = (k,E)$ is \textbf{minimum} if $\lambda(P-i) < \lambda(P)$ for all $i\in [k]$, 
where $P-i$ denotes the new pattern obtained from $P$ by removing $i$ from $[k]$ and removing all $r$-multisets containing $i$ from $E$.
Note that the Lagrangian of a pattern is a generalization
of the well-known \textbf{hypergraph Lagrangian} (see e.g.~\cite{BT11,FR84}) that
has been successfully applied to Tur\'an-type problems, with
the basic idea going back to Motzkin and Straus~\cite{MS65}.

\textbf{Remark.}
The notion of pattern was introduced by Pikhurko in~\cite{PI14} to study the general properties of nondegenerate hypergraph Tur\'{a}n problems, and it was also used very recently in~\cite{LP22,LP22a}. 
Note that the definition of pattern in~\cite{PI14} is more general by allowing recursive parts. 
Our results about patterns in this work can be easily extended to this more general setting.

Let $F$ be an $r$-graph and $P$ be a pattern. 
We say $(F,P)$ is a \textbf{Tur\'{a}n pair} if every $P$-construction is $F$-free 
and every maximum $F$-free construction is a $P$-construction. 
For example, it follows from the Tur\'{a}n Theorem that $\left(K_{\ell+1}, K_{\ell}\right)$ is a Tur\'{a}n pair for all $\ell \ge 2$.
It is easy to observe that for a Tur\'{a}n pair $(F,P)$, we have
\begin{align}\label{equ:Turan-density-Lagrangian}
    \pi(F) = \lambda(P). 
\end{align}

For hypergraphs in Tur\'{a}n pairs, we have the following result concerning the smoothness of their Tur\'{a}n functions. 

\begin{theorem}\label{THM:Turan-pair-smooth}
    Suppose that $F$ is an $r$-graph and $P$ is a minimal pattern  such that $(F,P)$ is a Tur\'{a}n pair. 
    Then $\mathrm{ex}(n, F)$ is $4\binom{n-1}{r-2}$-smooth. 
\end{theorem}

The boundedness of $F$ is closely related to the stability of $F$. 
So we introduce some definitions related to stability. 
Suppose that $(F,P)$ is a Tur\'{a}n pair. 
\begin{itemize}
    \item We say $F$ is \textbf{edge-stable} with respect to $P$ if for every $\delta>0$ there exist constants $N_0$ and $\zeta >0$ such that for every $F$-free $r$-graph $\mathcal{H}$ on $n \ge N_0$ vertices with at least $\left(\pi(F)-\zeta\right)\binom{n}{r}$ edges, there exists a subgraph $\mathcal{H}'\subset \mathcal{H}$ with at least $\left(\pi(F)-\delta\right)\binom{n}{r}$ edges such that $\mathcal{H}'$ is a $P$-subconstruction. 
    \item We say $F$ is \textbf{vertex-extendable} with respect to $P$ if there exist constants $N_0$ and $\zeta >0$ such that for every $F$-free $r$-graph $\mathcal{H}$ on $n \ge N_0$ vertices satisfing $\delta(\mathcal{H}) \ge \left(\pi(F)-\zeta\right)\binom{n-1}{r-1}$ the following holds: 
    if $\mathcal{H}-v$ is a $P$-subconstruction for some vertex $v\in V(\mathcal{H})$, then $\mathcal{H}$ is also a $P$-subconstruction.
    \item We say $F$ is \textbf{weakly vertex-extendable} with respect to $P$ if for every $\delta>0$ there exist constants $N_0$ and $\zeta >0$ such that for every $F$-free $r$-graph $\mathcal{H}$ on $n \ge N_0$ vertices satisfying $\delta(\mathcal{H}) \ge \left(\pi(F)-\zeta\right)\binom{n-1}{r-1}$ the following holds: 
    if $\mathcal{H}-v$ is a $P$-subconstruction for some vertex $v\in V(\mathcal{H})$, then $d_{\mathcal{H}}(v) \le \left(\pi(F)+\delta\right)\binom{n-1}{r-1}$. 
\end{itemize}
For simplicity, 
if $P$ is clear from the context, we will simply say that $F$ is edge-stable, vertex-extendable, and weakly vertex-extendable, respectively.

The first stability theorem which states that $K_{\ell+1}$ is edge-stable with respect to $K_{\ell}$ was proved independently by Erd\H{o}s and Simonovits~\cite{SI68}, and it was used first by Simonovits~\cite{SI68} to determine the exact Tur\'{a}n number $\mathrm{ex}(n,F)$ of an edge-critical graph $F$ for large $n$. 
Later, Simonovits' method (also known as the Stability Method) was used by many researchers to determine the Tur\'{a}n numbers of a large collection of hypergraphs (see Section~\ref{SEC:Applications} for more details).

The definition of vertex-extendability was introduced by Mubayi, Reiher, and the third author in~\cite{LMR23unif} for a unified framework for proving the stability of a large class of hypergraphs. 

The definition of weak vertex-extendability seems to be new, and it is clear from (\ref{equ:Turan-density-Lagrangian}) and the following lemma that for a Tur\'{a}n pair $(F,P)$ the vertex-extendability implies the weak vertex-extendability. 
There are several examples showing that the inverse is not true in general (see e.g Section~\ref{SUBSEC:expanded-triangle}). 
It seems interesting to explore the relations between the weak vertex-extendability and other types of stability (see~\cite{LMR23unif} for more details). 

\begin{lemma}[{\cite[Lemma~21]{LP22}}]\label{LEMMA:LP-pattern-min-max}
    Suppose that $P$ is a minimal pattern. 
    Then for every $\delta>0$ there exist $N_0$ and $\varepsilon>0$ such that every $P$-subconstruction $\mathcal{H}$ on $n \ge N_0$ vertices with $\delta(\mathcal{H}) \ge \left(\lambda(P)-\varepsilon\right)\binom{n-1}{r-1}$ satisfies $\Delta(\mathcal{H}) \le \left(\lambda(P)+ \delta\right)\binom{n-1}{r-1}$. 
\end{lemma}

Let us add another remark about the weak vertex-extendability that might be useful
for readers who are familiar with the stability method. 
In a standard stability argument in determining the exact value of $\mathrm{ex}(n,F)$, one usually defines a set $\mathcal{B}$ of bad edges and a set $\mathcal{M}$ of missing edges, and then tries to prove that $|\mathcal{M}| > |\mathcal{B}|$. 
One key step in this argument is to prove that the maximum degree of $\mathcal{B}$ is small (more specifically, $\Delta(B) = o(n^{r-1})$), which, informally speaking, usually implies the weak vertex-extendability of $F$.

For a Tur\'{a}n pair $(F,P)$ with the weak vertex-extendability, we have the following result concerning the boundedness of $F$. 

\begin{theorem}\label{THM:vertex-extend-bounded}
    Suppose that $F$ is an $r$-graph and $P$ is a minimal pattern such that $F$ is edge-stable and weakly vertex-extendable (or vertex-extendable) with respect to $P$. Then there exists a constant $c >0$ such that 
    $F$ is $\left(c \binom{n-1}{r-1}, \frac{1-\pi(F)}{8m} \binom{n-1}{r-1}\right)$-bounded for large $n$. 
\end{theorem}
\textbf{Remark.}
It seems possible to extend
Theorems~\ref{THM:Turan-pair-smooth} and~\ref{THM:vertex-extend-bounded} to nonminimal patterns, 
but we do not aware of any $r$-graph $F$ whose extremal construction is a $P$-construction for some nonminimal pattern $P$.
However, there does exist a finite family $\mathcal{F}$ of $r$-graphs whose extremal construction is a $P$-construction for some nonminimal pattern $P$ (see~\cite{HLLMZ22} for more details).
 
In many cases, (weak) vertex-extendability of $F$ follows from a stronger type of stability that was studied by many researchers before. 
Suppose that $(F,P)$ is a Tur\'{a}n pair.
We say $F$ is \textbf{degree-stable} with respect to $P$ if there exists $\zeta>0$ such that
for large $n$ every $n$-vertex $F$-free $r$-graph $\mathcal{H}$ with $\delta(\mathcal{H}) \ge \left(\pi(F) - \zeta\right) \binom{n-1}{r-1}$ is a $P$-subconstruction.  
It is easy to observe from the definition that if $F$ is degree-stable with respect to $P$, 
then $F$ is edge-stable and vertex-extendable with respect to $P$. 
Therefore, we have the following corollary of Theorems~\ref{THM:Turan-pair-smooth} and~\ref{THM:vertex-extend-bounded}.

\begin{corollary}\label{CORO:degree-stable-smooth-bounded}
    Suppose that $F$ is an $r$-graph and $P$ is a minimal pattern such that $F$ is degree-stable with respect to $P$. Then there exists a constant $c >0$ such that 
    \begin{enumerate}[label=(\alph*)]
        \item  
            $\mathrm{ex}(n,F)$ is $4\binom{n-1}{r-2}$-smooth, and 
        \item
            $F$ is $\left(c \binom{n-1}{r-1}, \frac{1-\pi(F)}{8m} \binom{n-1}{r-1}\right)$-bounded. 
    \end{enumerate}
\end{corollary}

In the next section, we show some applications of Theorems~\ref{THM:Main-simple},~\ref{THM:Turan-pair-smooth} and~\ref{THM:vertex-extend-bounded}, and Corollary~\ref{CORO:degree-stable-smooth-bounded}. 
We omit the applications of Theorem~\ref{THM:main-rainbow} since they are quite straightforward to obtain once we present the corresponding applications of Theorem~\ref{THM:Main-simple}. 
The proofs for Theorems~\ref{THM:Main-simple} and~\ref{THM:main-rainbow} are included in Section~\ref{SEC:Proof-Main}.  
The proofs for Theorems~\ref{THM:Turan-pair-smooth} and~\ref{THM:vertex-extend-bounded} are included in Section~\ref{SEC:Proof-applications}. 
\section{Applications}\label{SEC:Applications}
Combining some known stability results with Theorems~\ref{THM:Main-simple},~\ref{THM:Turan-pair-smooth}, and~\ref{THM:vertex-extend-bounded} (or Corollary~\ref{CORO:degree-stable-smooth-bounded}) we can immediately obtain results in this section. 
To demonstrate a way to apply Theorems~\ref{THM:Main-simple},~\ref{THM:Turan-pair-smooth}, and~\ref{THM:vertex-extend-bounded} in general, 
we include the short proof for the weak vertex-extendability of $\mathbb{F}_{3,2}$ (even though it can be deduced from results in~\cite{FPS053Book3page}). 

\subsection{Edge-critical graphs}\label{SUBSEC:edge-critical-graphs}
Recall that for a graph $F$ its chromatic number is denoted by $\chi(F)$.
We say a graph $F$ is \textbf{edge-critical} if there exists an edge $e\in F$ such that $\chi(F-e) < \chi(F)$.
Using the stability method, Simonovits proved in~\cite{SI68} that if a graph $F$ is edge-critical and $\chi(F) \ge 3$, then $\mathrm{EX}(n,F) = \{T(n, \chi(F)-1)\}$ for all sufficiently large $n$. 

Extending the classical Andr\'{a}sfai--Erd\H{o}s--S\'{o}s Theorem~\cite{AES74}, 
Erd\H{o}s and Simonovits~\cite{ES73} proved that every edge-critical graph with chromatic number at least $3$ is degree-stable. 
Theorefore, combined with Theorem~\ref{THM:Main-simple} and Corollary~\ref{CORO:degree-stable-smooth-bounded}, we obtain the following result.

\begin{theorem}\label{THM:edge-critical}
    Suppose that $F$ is an edge-critical graph with $\chi(F) \ge 3$.
    Then there exist constants $N_0$ and $c_F>0$ such that for all integers $n\ge N_0$ and $t \in [0, c_F n]$ we have 
    \begin{align*}
        \mathrm{EX}(n, (t+1)F) = \left\{K_{t}\uproduct T(n-t, \chi(F)-1)\right\}. 
    \end{align*}
\end{theorem}

\textbf{Remarks.}
\begin{itemize}
    \item For Theorem~\ref{THM:edge-critical} and all other theorems in this section, we did not try to optimize the constant $c_F$, but it seems possible to obtain a reasonable bound\footnote{It seems possible to get a polynomial dependency between $c_F$ and $\frac{1}{rm}$.} for $c_F$ by a more careful analysis of the proof for Theorem~\ref{THM:vertex-extend-bounded} (and the proof for the (weak) vertex-extendability of $F$ in some cases). 
    \item The case when $F$ is an odd cycle was also considered in a recent paper~{\cite[Theorem~1.1]{LMH23}}. 
    \item It might be true that Theorem~\ref{THM:edge-critical} holds for a broader class of graphs, and it would be interesting to characterize the class of graphs for which Theorem~\ref{THM:edge-critical} holds. 
\end{itemize}

\subsection{The Fano plane}\label{SUBSEC:Fano}
The \textbf{Fano plane} $\mathbb{F}$ is a $3$-graph with vertex set $\{1,2,3,4,5,6,7\}$ and edge set
\begin{align*}
    \{123,345,561,174,275,376,246\}. 
\end{align*}
Let $[n] = V_1 \cup V_2$ be a partition with $|V_1| = \lfloor n/2 \rfloor$ and $|V_2| = \lceil n/2 \rceil$. 
Let $B_3(n)$ denote the $3$-graph on $[n]$ whose edge set consists of all triples that have a nonempty intersection with both $V_1$ and $V_2$.  
Note that $|B_3(n)| \sim \frac{3}{4}\binom{n}{3}$. 

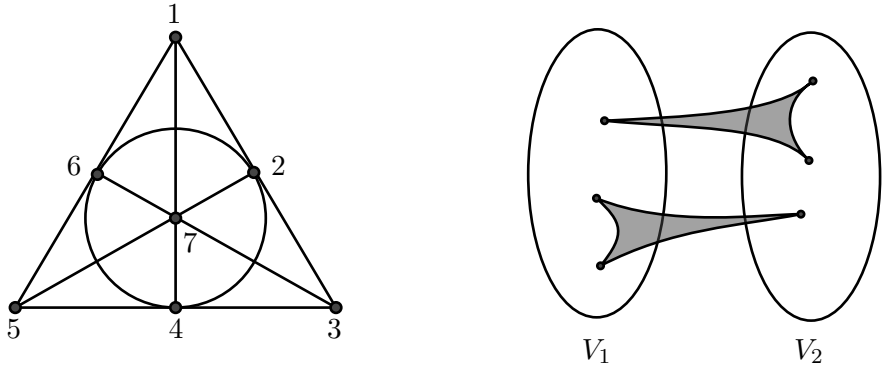
\begin{figure}[htbp]
\centering

\tikzset{every picture/.style={line width=1pt}} 

\begin{tikzpicture}[x=0.75pt,y=0.75pt,yscale=-1,xscale=1,line join=round]

\draw   (389.01,216.91) .. controls (369.98,216.81) and (354.71,184.29) .. (354.92,144.28) .. controls (355.12,104.27) and (370.71,71.91) .. (389.74,72.01) .. controls (408.77,72.11) and (424.03,104.62) .. (423.83,144.63) .. controls (423.63,184.64) and (408.04,217) .. (389.01,216.91) -- cycle ;
\draw   (496.07,218.55) .. controls (477.04,218.56) and (461.59,186.13) .. (461.57,146.12) .. controls (461.55,106.11) and (476.96,73.66) .. (495.99,73.65) .. controls (515.02,73.64) and (530.47,106.07) .. (530.49,146.08) .. controls (530.51,186.1) and (515.1,218.54) .. (496.07,218.55) -- cycle ;
\draw [fill=uuuuuu, fill opacity=0.5]  (393,118) .. controls (431,116) and (479,115) .. (497,98)  .. controls (480,110) and (484,127) .. (495,138) .. controls (478,120) and (434,122) .. (393,118) -- cycle;
\draw  [fill=uuuuuu, fill opacity=0.5]  (389,157) .. controls (425,172) and (472,165) .. (491,165) .. controls (460,171) and (427,172) .. (391,191) .. controls (404,175) and (402,169) .. (389,157) --cycle;
\draw   (134,167) .. controls (134,142.15) and (154.15,122) .. (179,122) .. controls (203.85,122) and (224,142.15) .. (224,167) .. controls (224,191.85) and (203.85,212) .. (179,212) .. controls (154.15,212) and (134,191.85) .. (134,167) -- cycle ;
\draw   (179,76) -- (259,212) -- (99,212) -- cycle ;
\draw    (179,76) -- (179,212) ;
\draw    (99,212) -- (218,144) ;
\draw    (141,145) -- (259,212) ;
\draw [fill=uuuuuu] (393,118) circle (1.2pt);
\draw [fill=uuuuuu]  (389,157) circle (1.2pt);
\draw [fill=uuuuuu] (491,165) circle (1.2pt);
\draw [fill=uuuuuu] (391,191) circle (1.2pt);
\draw [fill=uuuuuu] (497,98) circle (1.2pt);
\draw [fill=uuuuuu] (495,138) circle (1.2pt);
\draw (380,226) node [anchor=north west][inner sep=0.75pt]   [align=left] {$V_1$};
\draw (486,226) node [anchor=north west][inner sep=0.75pt]   [align=left] {$V_2$};
\draw [fill=uuuuuu] (179,76) circle (2pt);
\draw (173,58) node [anchor=north west][inner sep=0.75pt]   [align=left] {$1$};
\draw [fill=uuuuuu] (218,144) circle (2pt);
\draw (225,134) node [anchor=north west][inner sep=0.75pt]   [align=left] {$2$};
\draw [fill=uuuuuu] (259,212) circle (2pt);
\draw (253,216) node [anchor=north west][inner sep=0.75pt]   [align=left] {$3$};
\draw [fill=uuuuuu] (179,212) circle (2pt);
\draw (174,216) node [anchor=north west][inner sep=0.75pt]   [align=left] {$4$};
\draw [fill=uuuuuu] (99,212) circle (2pt);
\draw (93,216) node [anchor=north west][inner sep=0.75pt]   [align=left] {$5$};
\draw [fill=uuuuuu] (140,145) circle (2pt);
\draw (123,134) node [anchor=north west][inner sep=0.75pt]   [align=left] {$6$};
\draw [fill=uuuuuu] (179,167) circle (2pt);
\draw (181,173) node [anchor=north west][inner sep=0.75pt]   [align=left] {$7$};

\end{tikzpicture}

\caption{The Fano plane and the complete bipartite $3$-graph $B_3(n)$.}
\label{fig:Fano}
\end{figure}

It was conjectured by S\'{o}s~\cite{Sos76} and famously proved by De Caen and F\"uredi~\cite{DF00} that $\pi(\mathbb{F}) = 3/4$. 
Later, using a stability argument, Keevash and Sudakov~\cite{KS05}, and independently, F\"uredi and Simonovits~\cite{FS05} proved that $\mathrm{EX}(n, \mathbb{F}) = \{B_3(n)\}$ for all sufficienly large $n$. 
Recently, Bellmann and Reiher~\cite{BR19} proved that $\mathrm{ex}(n, \mathbb{F}) = |B_3(n)| = \frac{n-2}{2}\lfloor \frac{n^2}{4}\rfloor$ for all $n \ge 7$, and moreover, they proved that $B_3(n)$ is the unique extremal construction for all $n \ge 8$. 

It follows from the result of Keevash and Sudakov~\cite{KS05}, and independently, F\"uredi and Simonovits~\cite{FS05} that $\mathbb{F}$ is degree-stable. Therefore, we obtain the following result. 

\begin{theorem}\label{THEOREM:Fano}
There exist constants $N_0$ and $c_{\mathbb{F}}>0$ such that for all integers $n\ge N_0$ and $t \in [0, c_{\mathbb{F}} n]$ we have 
\begin{align*}
    \mathrm{EX}(n, (t+1)\mathbb{F}) = \left\{K_{t}^{3} \uproduct B_3(n-t)\right\}.
\end{align*}
\end{theorem}

\subsection{Generalized triangles}\label{SUBSEC:generalized-triangle}
The ($r$-uniform) \textbf{generalized triangle} $\mathds{T}_r$ is the $r$-graph with vertex set $[2r-1]$ and edge set  
\begin{align*}
    \left\{\{1,\ldots,r-1, r\}, \{1,\ldots, r-1, r+1\}, \{r,r+1, \ldots, 2r-1\}\right\}. 
\end{align*}
Note that $\mathds{T}_2$ is simply a triangle. 

Fix $n \ge r \ge 2$ and $\ell \ge r$. 
Let $[n] = V_1 \cup \cdots \cup V_{\ell}$ be a partition such that $|V_i| \in \left\{\lfloor \frac{n}{\ell}\rfloor, \lceil \frac{n}{\ell} \rceil\right\}$ for all $i\in [\ell]$. 
The \textbf{generalized Tur\'{a}n $r$-graph} $T_{r}(n,\ell)$ is the $r$-graph on $[n]$ whose edge set consists of all $r$-sets that contain at most one vertex from each $V_i$. 
Note that $T_2(n,\ell)$ is the Tur\'{a}n graph $T(n, \ell)$. 
Let $t_{r}(n,\ell)$ denote the number of edges in $T_{r}(n,\ell)$. 

\begin{figure}[htbp]
\centering

\tikzset{every picture/.style={line width=1pt}} 

\begin{tikzpicture}[x=0.75pt,y=0.75pt,yscale=-1,xscale=1,line join=round]

\draw   (417,88.5) .. controls (417,67.79) and (433.79,51) .. (454.5,51) .. controls (475.21,51) and (492,67.79) .. (492,88.5) .. controls (492,109.21) and (475.21,126) .. (454.5,126) .. controls (433.79,126) and (417,109.21) .. (417,88.5) -- cycle ;
\draw   (480,178.5) .. controls (480,157.79) and (496.79,141) .. (517.5,141) .. controls (538.21,141) and (555,157.79) .. (555,178.5) .. controls (555,199.21) and (538.21,216) .. (517.5,216) .. controls (496.79,216) and (480,199.21) .. (480,178.5) -- cycle ;
\draw   (356,178.5) .. controls (356,157.79) and (372.79,141) .. (393.5,141) .. controls (414.21,141) and (431,157.79) .. (431,178.5) .. controls (431,199.21) and (414.21,216) .. (393.5,216) .. controls (372.79,216) and (356,199.21) .. (356,178.5) -- cycle ;
\draw [fill=uuuuuu, fill opacity=0.5]   (159.54,93.17) .. controls (173.12,108.44) and (175.21,133.88) .. (164.76,151.69) .. controls (196.1,112.25) and (219.08,94.44) .. (258.78,69) .. controls (229.53,84.27) and (195.06,89.35) .. (159.54,93.17) -- cycle;
\draw [fill=uuuuuu, fill opacity=0.5]   (159.54,93.17) .. controls (194.01,107.16) and (230.57,128.79) .. (262.96,131.33) .. controls (222.22,131.33) and (197.14,130.06) .. (164.76,151.69)  .. controls (175.21,133.88) and (173.12,108.44) .. (159.54,93.17) -- cycle;
\draw [fill=uuuuuu, fill opacity=0.5]   (258.78,69) .. controls (276.54,77.9) and (275.5,131.33) .. (262.96,131.33) .. controls (280.72,140.24) and (276.54,202.57) .. (264,202.57) .. controls (295.34,191.12) and (283.85,77.9) .. (258.78,69) -- cycle;
\draw  [fill=uuuuuu, fill opacity=0.5]   (454.5,88.5) .. controls (465,132) and (483,161) .. (517.5,178.5) .. controls (481,158) and (441,156) .. (393.5,178.5) .. controls (423,158) and (456,121) .. (454.5,88.5) -- cycle;
\draw [fill=uuuuuu] (159.54,93.17) circle (2pt);
\draw [fill=uuuuuu] (164.76,151.69) circle (2pt);
\draw [fill=uuuuuu] (258.78,69) circle (2pt);
\draw [fill=uuuuuu] (262.96,131.33) circle (2pt);
\draw [fill=uuuuuu] (264,202.57) circle (2pt);
\draw [fill=uuuuuu] (517.5,178.5) circle (1.2pt);
\draw [fill=uuuuuu] (454.5,88.5) circle (1.2pt);
\draw [fill=uuuuuu] (393.5,178.5) circle (1.2pt);
\draw (145,86) node [anchor=north west][inner sep=0.75pt]   [align=left] {$1$};
\draw (149,144) node [anchor=north west][inner sep=0.75pt]   [align=left] {$2$};
\draw (254,74) node [anchor=north west][inner sep=0.75pt]   [align=left] {$3$};
\draw (257,136) node [anchor=north west][inner sep=0.75pt]   [align=left] {$4$};
\draw (258,208) node [anchor=north west][inner sep=0.75pt]   [align=left] {$5$};
\draw (443.99,29.21) node [anchor=north west][inner sep=0.75pt]   [align=left] {$V_1$};
\draw (514.99,222.21) node [anchor=north west][inner sep=0.75pt]   [align=left] {$V_2$};
\draw (379.99,221.21) node [anchor=north west][inner sep=0.75pt]   [align=left] {$V_3$};

\end{tikzpicture}

\caption{The generealized triangle $\mathds{T}_3$ and the Tur\'{a}n $3$-graph $T_{3}(n,3)$.}
\label{fig:T3}
\end{figure}
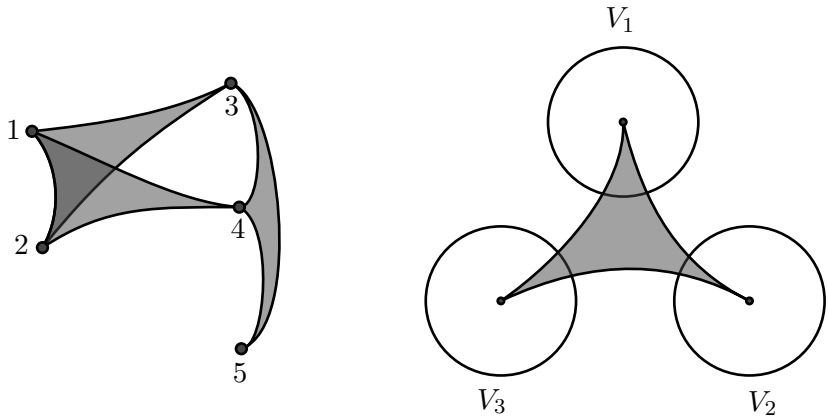

Katona conjectured and Bollob\'{a}s~\cite{BO74} proved that 
$\mathrm{EX}(n, \{\mathds{T}_3, K_{4}^{3-}\}) = \{T_3(n,3)\}$ for all $n\in \mathbb{N}$, 
where $K_{4}^{3-}$ is the unique $3$-graph with $4$ vertices and $3$ edges. 
Later, Frankl and F\"{u}redi~\cite{FF83} sharpened the result of Bollob\'{a}s by showing that $\mathrm{EX}(n, \mathds{T}_3) = \{T_3(n,3)\}$ for all $n\ge 3000$. 
In~\cite{KM04}, Keevash and Mubayi proved the edge-stability of $\mathds{T}_3$ and improved the lower bound of $n$ from $3000$ to $33$.
A short proof for the edge-stability with a linear dependency between the error parameters can be found in~\cite{LIU19}. 

The vertex-extendability of $\mathds{T}_3$ can be easily obtained from the proof of Lemma~4.4 in~\cite{LMR23unif} (also see the Concluding Remarks in~\cite{LMR23unif}).  
Therefore, we obtain the following result. 

\begin{theorem}\label{THEOREM:T_3}
There exist constants $N_0$ and $c_{\mathds{T}_3}$ such that for all integers $n \ge N_0$ and $t \in [0, c_{\mathds{T}_3}n]$ we have 
\begin{align*}
    \mathrm{EX}(n, (t+1)\mathds{T}_3) = \left\{K_{t}^{3}\uproduct T_3(n-t,3)\right\}. 
\end{align*}
\end{theorem}

For $r = 4$, improving a result of Sidorenko in~\cite{Sido87}, 
Pikhurko proved in~\cite{PI08} that 
$\mathrm{EX}(n,\mathds{T}_4) = \left\{T_4(n,4)\right\}$ for all sufficiently large $n$.

Similarly, the vertex-extendability of $\mathds{T}_4$ can be obtained from the proof of Lemma~4.4 in~\cite{LMR23unif} (also see the Concluding Remarks in~\cite{LMR23unif}).  
Therefore, we obtain the following result. 

\begin{theorem}\label{THEOREM:T_4}
There exist constants $N_0$ and $c_{\mathds{T}_4}$ such that for all integers $n \ge N_0$ and $t \in [0, c_{\mathds{T}_4}n]$ we have 
\begin{align*}
    \mathrm{EX}(n, (t+1)\mathds{T}_4) = \left\{K_{t}^{4}\uproduct T_4(n-t,4)\right\}. 
\end{align*}
\end{theorem}

The situation becomes complicated when $r\ge 5$. 
Let $\mathds{W}_5$ denote the unique $5$-graph with $11$ vertices such that every $4$-set of vertices is contained in exactly one edge. 
Let $\mathds{W}_6$ denote the unique $6$-graph with $12$ vertices such that every $5$-set of vertices is contained in exactly one edge. 
Let $\mathds{W}_5(n)$ and $\mathds{W}_6(n)$ denote the maximum $\mathds{W}_5$-construction and $\mathds{W}_6$-construction on $n$ vertices, respectively. 
Some calculations show that $\mathds{W}_5(n) \sim \frac{6}{11^4}n^5$ and $\mathds{W}_{6}(n) \sim \frac{11}{12^5}n^6$. 

In~\cite{FF89}, Frankl and F\"{u}redi proved that
$\mathrm{ex}(n,\mathds{T}_r) \le |\mathds{W}_r(n)|+o(n^r)$ for $r =5,6$. 
Much later, using a sophisticated stability argument, Norin and Yepremyan~\cite{NY17} proved that $\mathds{T}_5$ and $\mathds{T}_6$ are edge-stable with respect to $\mathds{W}_5$ and $\mathds{W}_6$ respectively, and moreover, 
$\mathrm{EX}(n,\mathds{T}_r) = \{\mathds{W}_r(n)\}$ for $r =5,6$ and large $n$.

It was observed by Pikhurko~\cite{PI08} that both $\mathds{T}_5$ and $\mathds{T}_6$ fail to be degree-stable (or vertex-extendable). 
However, from Lemmas~7.2 and~7.4 in~\cite{NY17} one can easily observe that $\mathds{T}_5$ and $\mathds{T}_6$ are weakly vertex-extendable. 
Therefore, we obtain the following theorem. 

\begin{theorem}\label{THEOREM:T_5-T-6}
For $r\in \{5,6\}$ there exist constants $N_0$ and $c_{\mathds{T}_r}>0$ such that for all integers $n \ge N_0$ and $t \in [0, c_{\mathds{T}_r} n]$ we have 
\begin{align*}
    \mathrm{EX}(n, (t+1)\mathds{T}_r) = \left\{K_{t}^{r}\uproduct \mathds{W}_r(n-t)\right\}. 
\end{align*}
\end{theorem}

It seems that there are even no conjectures for the extremal constructions of $\mathds{T}_r$ when $r\ge 7$.
We refer the reader to~\cite{FF89} for some lower and upper bounds for $\pi(\mathds{T}_r)$ in general. 
\subsection{The expansion of complete graphs}\label{SUBSEC:graph-expansion}
Fix integers $\ell \ge r \ge 2$. 
The \textbf{expansion} $H_{\ell+1}^{r}$ of the complete graph $K_{\ell+1}$ is the $r$-graph obtained from $K_{\ell+1}$ by adding a set of $r-2$ new vertices into each edge of $K_{\ell+1}$, and moreover, these new $(r-2)$-sets are pairwise disjoint. 
It is clear from the definition that $H_{\ell+1}^{r}$ has $\ell+1+(r-2)\binom{\ell+1}{2}$ vertices and $\binom{\ell+1}{2}$ edges.

\begin{figure}[htbp]
\centering

\tikzset{every picture/.style={line width=1pt}} 

\begin{tikzpicture}[x=0.75pt,y=0.75pt,yscale=-1,xscale=1,line join=round]

\draw   (429,90.5) .. controls (429,69.79) and (445.79,53) .. (466.5,53) .. controls (487.21,53) and (504,69.79) .. (504,90.5) .. controls (504,111.21) and (487.21,128) .. (466.5,128) .. controls (445.79,128) and (429,111.21) .. (429,90.5) -- cycle ;
\draw   (492,180.5) .. controls (492,159.79) and (508.79,143) .. (529.5,143) .. controls (550.21,143) and (567,159.79) .. (567,180.5) .. controls (567,201.21) and (550.21,218) .. (529.5,218) .. controls (508.79,218) and (492,201.21) .. (492,180.5) -- cycle ;
\draw   (368,180.5) .. controls (368,159.79) and (384.79,143) .. (405.5,143) .. controls (426.21,143) and (443,159.79) .. (443,180.5) .. controls (443,201.21) and (426.21,218) .. (405.5,218) .. controls (384.79,218) and (368,201.21) .. (368,180.5) -- cycle ;
\draw  [fill=uuuuuu, fill opacity=0.5]  (466.5,90.5) .. controls (477,134) and (495,163) .. (529.5,180.5) .. controls (493,160) and (453,158) .. (405.5,180.5) .. controls (435,160) and (468,123) .. (466.5,90.5) -- cycle;
\draw   (196,61) -- (291,215) -- (101,215) -- cycle ;
\draw    (196,61) -- (197,159) ;
\draw    (197,159) -- (101,215) ;
\draw    (197,159) -- (291,215) ;
\draw [fill=uuuuuu] (196,61) circle (2pt);
\draw [fill=uuuuuu] (291,215) circle (2pt);
\draw [fill=uuuuuu] (101,215) circle (2pt);
\draw [fill=uuuuuu] (197,159) circle (2pt);
\draw [fill={rgb:red,0;green,0;blue,1}] (234,181) circle (3pt);
\draw [fill={rgb:red,0;green,0;blue,1}] (161,181) circle (3pt);
\draw [fill={rgb:red,0;green,0;blue,1}] (197,120) circle (3pt);
\draw [fill={rgb:red,0;green,0;blue,1}] (199,215) circle (3pt);
\draw [fill={rgb:red,0;green,0;blue,1}] (243,136) circle (3pt);
\draw [fill={rgb:red,0;green,0;blue,1}] (150,136) circle (3pt);
\draw [fill=uuuuuu] (466.5,90.5) circle (1.2pt);
\draw [fill=uuuuuu] (529.5,180.5) circle (1.2pt);
\draw [fill=uuuuuu] (405.5,180.5) circle (1.2pt);
\draw (455.99,31.21) node [anchor=north west][inner sep=0.75pt]   [align=left] {$V_1$};
\draw (526.99,224.21) node [anchor=north west][inner sep=0.75pt]   [align=left] {$V_2$};
\draw (391.99,223.21) node [anchor=north west][inner sep=0.75pt]   [align=left] {$V_3$};
\draw (191,42) node [anchor=north west][inner sep=0.75pt]   [align=left] {$1$};
\draw (293,218) node [anchor=north west][inner sep=0.75pt]   [align=left] {$2$};
\draw (87,218) node [anchor=north west][inner sep=0.75pt]   [align=left] {$3$};
\draw (191,170) node [anchor=north west][inner sep=0.75pt]   [align=left] {$4$};

\end{tikzpicture}

\caption{The expansion $H_{4}^3$ of $K_4$ and the Tur\'{a}n $3$-graph $T_{3}(n,3)$.}
\label{fig:H43}
\end{figure}
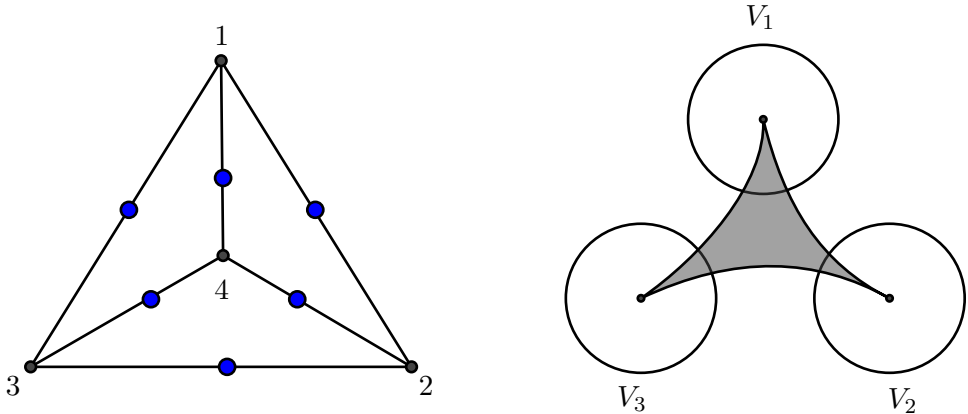

The $r$-graph $H_{\ell+1}^{r}$ was introduced by  Mubayi~\cite{MU06} as a way to generalize Tur\'{a}n’s theorem to hypergraphs. 
These hypergraphs provide the first explicitly defined examples 
which yield an infinite family of numbers realizable as Tur\'{a}n densities for hypergraphs.
In~\cite{MU06}, Mubayi determined the Tur\'{a}n density of $H_{\ell+1}^{r}$ for all integers $\ell \ge r \ge 3$, and proved that $H_{\ell+1}^{r}$ is edge-stable.  
In~\cite{PI13}, Pikhurko refined Mubayi's result and proved that $\mathrm{EX}(n, H_{\ell+1}^{r}) = \{T_r(n,\ell)\}$ for all integers $\ell \ge r \ge 3$ when $n$ is sufficiently large.

The vertex-extendability of $H_{\ell+1}^{r}$ can be easily obtained by a small modification of the proof of Lemma~4.8 in~\cite{LMR23unif} (also see the Concluding Remarks in~\cite{LMR23unif}). 
Therefore, we obtain the following result. 

\begin{theorem}\label{THEOREM:H_l^r}
Fix integers $\ell \ge r \ge 2$. 
There exist constants $N_0$ and $c = c(\ell, r) >0$ such that for all integers $n \ge N_0$ and $t\in [0, c n]$ we have 
\begin{align*}
    \mathrm{EX}(n, (t+1)H_{\ell+1}^{r}) = \left\{K_{t}^{r} \uproduct T_{r}(n-t, \ell)\right\}. 
\end{align*}
\end{theorem}

\textbf{Remarks.}
The definition of expansion can be extended to all graphs as follows. 
Fix a graph $F$,  
let the $r$-graph $H_{F}^{r}$ be obtained from $F$ by adding a set of $r-2$ new vertices into each edge of $F$, and moreover, these new $(r-2)$-sets are pairwise disjoint. 
Similar to Theorem~\ref{THM:edge-critical}, one could obtain a corresponding result for the expansion of all edge-critical graphs. We omit its statement and proof here. 
\subsection{The expansion of hypergraphs}\label{SUBSEC:hypergraph-expansion}
Given an $r$-graph $F$ with $\ell+1$ vertices, 
the \textbf{expansion} $H^{F}_{\ell+1}$ of $F$ is the $r$-graph obtained from $F$ by adding, for every pair $\{u,v\}\subset V(F)$ that is not contained in any edge of $F$, an $(r-2)$-set of new vertices, and moreover, these $(r-2)$-sets are pairwise disjoint. 
It is easy to see that the expansion of the empty $r$-graph on $\ell+1$ vertices (here empty means that the edge set is empty) is the same as the expansion of the complete graph $K_{\ell+1}$ defined in the previous subsection. 
However, in general, these two definitions are different. 

Our first result in this subsection is about the expansion of the expanded trees. 
Given a tree $T$ on $k$ vertices, define the \textbf{$(r-2)$-expansion} $\mathrm{Exp}(T)$ of $T$ as 
\begin{align*}
    \mathrm{Exp}(T) := \left\{e\cup A \colon e \in T\right\}, 
\end{align*}
where $A$ is a set of $r-2$ new vertices that is disjoint from $V(T)$. 

Given a tree $T$ on $k$ vertices, we say $T$ is an \textbf{Erd\H{o}s--S\'{o}s tree} if it satisfies the famous Erd\H{o}s--S\'{o}s conjecture on trees. 
In other words, $T$ is contained in every graph with average degree more than $k-2$. 
In~\cite{Sido89}, Sidorenko proved that for large $k$, if $T$ is an Erd\H{o}s--S\'{o}s tree on $k$ vertices, then   
$\mathrm{ex}(n,H^{\mathrm{Exp}(T)}_{k+r-2}) \le t_{r}(n, k+r-3) + o(n^{r})$. 
Much later, Norin and Yepremyan~\cite{NY18}, and independently, Brandt, Irwin, and Jiang~\cite{BIJ17}, improved Sidorenko's result by showing that, under the same setting, $H^{\mathrm{Exp}(T)}_{k+r-2}$ is edge-stable with respect to $K_{k+r-3}^{r}$ and $\mathrm{EX}(n,H^{\mathrm{Exp}(T)}_{k+r-2}) = \{T_r(n,k+r-3)\}$ for large $n$. 
In fact, it follows easily from Lemmas~3.5 and~4.1 in~\cite{NY18} that $H^{\mathrm{Exp}(T)}_{k+r-2}$ is weakly vertex-extendable with respect to $K_{k+r-3}^{r}$. 
Hence, we obtain the following result. 

\begin{theorem}\label{THM:ES-tree}
    For every integer $r\ge 3$ there exists $M_{r}$ such that  
      if $T$ is an Erd\H{o}s--S\'{o}s tree on $k\ge M_r$ vertices, 
      then there exist $N_0$ and $c_T>0$ such that for all integers $n \ge N_0$ and $t \le c_T n$, we have 
      \begin{align*}
          \mathrm{EX}\left(n,(t+1)H^{\mathrm{Exp}(T)}_{k+r-2}\right)
          = K_{t}^{r} \uproduct T_r(n-t,k+r-3). 
      \end{align*}
\end{theorem}

Next, we consider the expansion of a different class of hypergraphs. 
Let $B(r, \ell+1)$ be the $r$-graph with vertex set $[\ell+1]$ and edge set 
\begin{align*}
    \left\{[r]\right\} 
    \cup \left\{e\subset [2, \ell+1] \colon |e|=r \text{ and } |e \cap [2,r]| \le 1\right\}. 
\end{align*}

Recall that the Lagrangian of an $r$-graph $\mathcal{H}$ (by viewing $\mathcal{H}$ as a pattern) is denoted by $\lambda(\mathcal{H})$.
%
For integers $\ell \ge r \ge 2$ let the family $\mathcal{F}_{\ell+1}^r$ be the collection of $r$-graphs $F$ with the following properties: 
\begin{enumerate}[label=(\alph*)]
    \item
        $\sup\left\{\lambda(\mathcal{H}) \colon\text{$\mathcal{H}$ is $F$-free and not a $K_{\ell}^r$-subconstruction}\right\} < \frac{\ell \cdots (\ell-r+1)}{\ell^r}$, and
    \item
        either $F$ has an isolated vertex or $F\subset B(r, \ell+1)$.  
\end{enumerate}

For every $F\in \mathcal{F}_{\ell+1}^r$ the vertex-extendability\footnote{The weak vertex-extendability of $F\in \mathcal{F}_{\ell+1}^r$ with an isolated vertex also follows from Lemma~3.4 in~\cite{NY18}.} of the expansion $H_{\ell+1}^{F}$ can be easily obtained by a small modification of the proof of Lemma~4.8 in~\cite{LMR23unif} (also see the Concluding Remarks in~\cite{LMR23unif}). 
Hence, we obtain the following result. 

\begin{theorem}\label{THM:Hypergraph-Expansion-A}
    Suppose that $\ell\ge r\ge 2$ are integers and $F\in \mathcal{F}_{\ell+1}^r$. 
    Then there exist constants $N_0$ and $c_F > 0$ such that for all integers $n \ge N_0$ and $t\in [0, c_F n]$, we have 
    \begin{align*}
        \mathrm{EX}\left(n, (t+1)H_{\ell+1}^{F}\right) = \left\{K_{t}^{r} \uproduct T_{r}(n-t,\ell)\right\}. 
    \end{align*}
\end{theorem}

\textbf{Remarks.}
\begin{itemize}
    \item In~\cite{MP07Fan}, Mubayi and Pikhurko considered the Tur\'{a}n problem for the $r$-graph $\mathrm{Fan}^r$ (the generalized Fan), which is the expansion of the $r$-graph on $r+1$ vertices with only one edge. It is easy to see that $\mathrm{Fan}^r$ is a member in $\mathcal{F}_{r+1}^{r}$. 
    \item  The Tur\'{a}n problem for the expansion of certain class of $r$-graphs (which is a proper subfamily of $\mathcal{F}_{\ell+1}^r$) were studied previously in~\cite{BIJ17} and~\cite{NY18}. 
    \item Let $M_{k}^{r}$ denote the $r$-graph consisting of $k$ vertex-disjoint edges (i.e. a matching of size $k$) and let $L_{k}^{r}$ denote the $r$-graph consisting of $k$ edges having one vertex, say $v$, in common, and every pair of edges interest only at $v$ (i.e. a $k$-edge sunflower with the center $v$). 
    By results in~\cite{HK13,JPW18}, if $F$ is isomorphic to $M_{k}^{3}$ (see~\cite{HK13} for $k=2$ and~\cite{JPW18} for $k \ge 3$), $L_{k}^{3}$ (see~\cite{JPW18}), or $L_{k}^4$ (see~\cite{JPW18}), where $k \ge 2$ is an integer, then $F$ is contained in $\mathcal{F}_{\ell+1}^r$. 
\end{itemize}

Now we focus on the expansion of $r$-uniform matching of size two with $r\ge 4$. 
We say an $r$-graph is \textbf{semibipartite} if its vertex set can be partitioned into two parts $V_1$ and $V_2$ such that every edge contains exactly one vertex in $V_1$. 
Let $S_r(n)$ denote the semibipartite $r$-graph on $n$ vertices with the maximum number of edges. 
Simply calculations show that $|S_r(n)| \sim \left(\frac{r-1}{r}\right)^{r-1}\binom{n}{r}$. 

Confirming a conjecture of Hefetz and Keevash~\cite{HK13},
Bene Watts, Norin, and Yepremyan~\cite{BNY19} showed that for $r\ge 4$, $\mathrm{EX}\left(n, H_{2r}^{M_{2}^{r}}\right) = \{S_r(n)\}$ for all sufficiently large $n$. 

The vertex-extendability\footnote{The weak vertex-extendability of $H_{2r}^{M_{2}^{r}}$ also follows from Theorem~3.2 in~\cite{BNY19}} of $H_{2r}^{M_{2}^{r}}$ can be easily obtained by a small modification of the proof of Lemma~4.12 in~\cite{LMR23unif} (also see the Concluding Remarks in~\cite{LMR23unif}). 
Hence we have the following result. 

\begin{theorem}\label{THM:Hypergraph-Expansion-B}
    For every integer $r \ge 4$,
    there exist constants $N_0$ and $c = c(r) > 0$ such that for all integers $n \ge N_0$ and $t\in [0, c n]$, we have 
    \begin{align*}
        \mathrm{EX}\left(n, (t+1)H_{2r}^{M_{2}^{r}}\right) = \left\{K_{t}^{r} \uproduct S_{r}(n-t)\right\}. 
    \end{align*}
\end{theorem}

\textbf{Remark.}
It is quite possible that Theorem~\ref{THM:Main-simple} applies to the expansion of other hypergraphs,  
for example, the $3$-graph defined in~\cite{YP22} which provides the first example of a single hypergraph whose Tur\'{a}n density is an irrational number.  
\subsection{Expanded triangles}\label{SUBSEC:expanded-triangle}
Let $\mathcal{C}^{2r}_{3}$ denote the $2r$-graph with vertex set $[3r]$ and edge set 
\begin{align*}
    \left\{\{1,\ldots, r, r+1, \ldots, 2r\}, \{r+1, \ldots, 2r, 2r+1, \ldots, 3r\}, \{1,\ldots, r, 2r+1, \ldots, 3r\}\right\}. 
\end{align*}
Let $[n]= V_1 \cup V_2$ be a partition such that $|V_1| = \lfloor n/2 \rfloor+m$. 
Let $B_{2r}^{\mathrm{odd}}(n,m)$ denote the $2r$-graph on $[n]$ whose edge set consists of all $2r$-sets that interest $V_1$ in odd number of vertices. 
Some calculations show that $\max_{m}|B_{2r}^{\mathrm{odd}}(n,m)|\sim \frac{1}{2}\binom{n}{2r}$. 
Let $B_{2r}^{\mathrm{odd}} = (2,E)$ denote the pattern such that $E$ consists of all $2r$-multisets that contain exactly odd number of $1$s. 
Note that $B_{2r}^{\mathrm{odd}}(n,m)$ is a $B_{2r}^{\mathrm{odd}}$-construction.

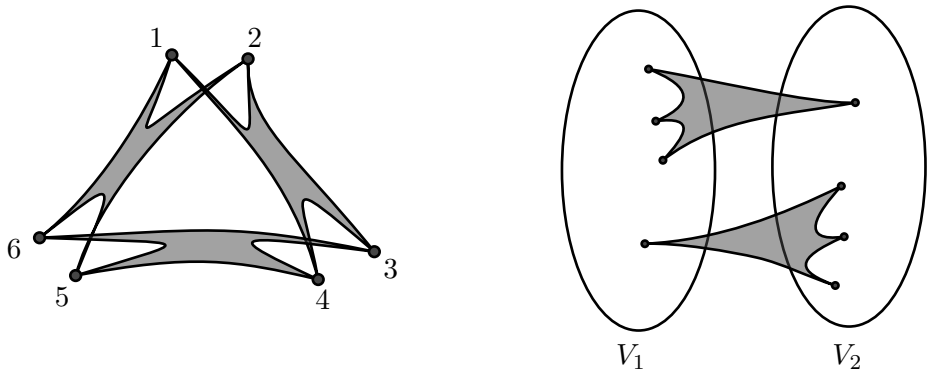
\begin{figure}[htbp]
\centering
\tikzset{every picture/.style={line width=1pt}} 

\begin{tikzpicture}[x=0.75pt,y=0.75pt,yscale=-1,xscale=1,line join=round]

\draw   (507.85,205.65) .. controls (486.76,205.66) and (469.64,169.64) .. (469.62,125.19) .. controls (469.6,80.75) and (486.68,44.71) .. (507.77,44.7) .. controls (528.86,44.69) and (545.98,80.71) .. (546,125.15) .. controls (546.02,169.6) and (528.94,205.64) .. (507.85,205.65) -- cycle ;
\draw   (402.85,207.65) .. controls (381.76,207.66) and (364.64,171.64) .. (364.62,127.19) .. controls (364.6,82.75) and (381.68,46.71) .. (402.77,46.7) .. controls (423.86,46.69) and (440.98,82.71) .. (441,127.15) .. controls (441.02,171.6) and (423.94,207.64) .. (402.85,207.65) -- cycle ;
\draw  [fill=uuuuuu, fill opacity=0.5]  (407.85,76.14) .. controls (424.42,82.69) and (435.22,91.97) .. (411.45,102.33) .. controls (432.34,99.6) and (426.32,109.74) .. (415,122) .. controls (443.09,99.63) and (479.31,99) .. (511,93) .. controls (489.4,94.09) and (435.22,80.51) .. (407.85,76.14) --cycle;
\draw  [fill=uuuuuu, fill opacity=0.5]  (501,185) .. controls (482.36,178.05) and (479.1,172.04) .. (505.43,160.45) .. controls (482.16,163.67) and (491.6,148.51) .. (504,135)  .. controls (474.51,149.61) and (442.47,162.24) .. (406,164)
 .. controls   (422.94,163.37) and (449,168) .. (462,171)  .. controls (475,174) and (487.72,180.88) .. (501,185) --cycle;
\draw  [fill=uuuuuu, fill opacity=0.5]   (170,69) .. controls (155,121) and (140,114) .. (208,71)  .. controls (171,96) and (132,143) .. (122,180) .. controls (143,128) and (144,130) .. (104,161) .. controls (132,139) and (156,101) .. (170,69) --cycle; 
\draw  [fill=uuuuuu, fill opacity=0.5]  (243,182) .. controls (235,137) and (222,126) .. (271,168) .. controls (231,118) and (207,104) .. (208,71) .. controls (208,121) and (218,116) .. (170,69) .. controls (201,105) and (230,131) .. (243,182) --cycle;
\draw  [fill=uuuuuu, fill opacity=0.5]  (243,182) .. controls (201,172) and (181,169) .. (122,180) .. controls (183,161) and (186,164) .. (104,161) .. controls (188,157) and  (199,153) .. (271,168) .. controls (193,158) and (197,160) .. (243,182) --cycle;
\draw [fill=uuuuuu]  (170,69)  circle (2pt);
\draw [fill=uuuuuu] (208,71) circle (2pt);
\draw [fill=uuuuuu] (122,180) circle (2pt);
\draw [fill=uuuuuu] (104,161) circle (2pt);
\draw [fill=uuuuuu] (243,182) circle (2pt);
\draw [fill=uuuuuu] (271,168) circle (2pt);

\draw [fill=uuuuuu]  (407.85,76.14) circle (1.2pt);
\draw [fill=uuuuuu] (411.45,102.33) circle (1.2pt);
\draw [fill=uuuuuu] (415,122) circle (1.2pt);
\draw [fill=uuuuuu]  (511,93) circle (1.2pt);
\draw [fill=uuuuuu] (501,185) circle (1.2pt);
\draw [fill=uuuuuu] (505.43,160.45) circle (1.2pt);
\draw [fill=uuuuuu] (504,135) circle (1.2pt);
\draw [fill=uuuuuu] (406,164) circle (1.2pt);
\draw (390,213) node [anchor=north west][inner sep=0.75pt]   [align=left] {$V_1$};
\draw (498,213) node [anchor=north west][inner sep=0.75pt]   [align=left] {$V_2$};
\draw (157,54) node [anchor=north west][inner sep=0.75pt]   [align=left] {$1$};
\draw (206,54) node [anchor=north west][inner sep=0.75pt]   [align=left] {$2$};
\draw (86,160) node [anchor=north west][inner sep=0.75pt]   [align=left] {$6$};
\draw (110,184) node [anchor=north west][inner sep=0.75pt]   [align=left] {$5$};
\draw (240,186) node [anchor=north west][inner sep=0.75pt]   [align=left] {$4$};
\draw (274,170) node [anchor=north west][inner sep=0.75pt]   [align=left] {$3$};

\end{tikzpicture}

\caption{The $4$-graph $\mathcal{C}_{3}^{4}$ (expanded triangle) and the $4$-graph $B_{4}^{\mathrm{odd}}(n)$.}
\label{fig:expandedK3}
\end{figure}

The Tur\'{a}n problem for $\mathcal{C}^{2r}_{3}$ was first considered by
Frankl~\cite{Frankl90}, who proved that $\pi(\mathcal{C}^{2r}_{3}) = 1/2$. 
Later, Keevash and Sudakov~\cite{KS05a} proved that $\mathcal{C}^{2r}_{3}$ is edge-stable with respect to $B_{2r}^{\mathrm{odd}}$, and moreover, $\mathrm{EX}(n,\mathcal{C}^{2r}_{3}) \subset \left\{B_{2r}^{\mathrm{odd}}(n,m)\colon m\in [0,n/2]\right\}$.
Simple constructions\footnote{For example, choose a set $S$ of $2r$ vertices from $V_1$ in $B_{2r}^{\mathrm{odd}}(n,0)$, then remove all edges in $B_{2r}^{\mathrm{odd}}(n,0)$ that contain at least two vertices in $S$ and add $S$ to the edge set.} show that $\mathcal{C}^{2r}_{3}$ is not degree-stable (or vertex-extendable) with respect to $B_{2r}^{\mathrm{odd}}$. 
However, using Claim~3.5 in~\cite{KS05a}, one can easily show that $\mathcal{C}^{2r}_{3}$ is weakly vertex-extendable with respect to $B_{2r}^{\mathrm{odd}}$.
Hence, we have the following theorem.

\begin{theorem}\label{THM:expanded-triangle}
    For every integer $r \ge 2$
    there exist constants $N_0$ and $c > 0$ such that for all integers $n \ge N_0$ and $t\in [0, c n]$, we have 
    \begin{align*}
        \mathrm{EX}\left(n, (t+1)\mathcal{C}^{2r}_{3}\right) \subset K_{t}^{2r} \uproduct \left\{B_{2r}^{\mathrm{odd}}(n-t,m)\colon m\in \left[0,\sqrt{2r(n-t)}\right]\right\}. 
    \end{align*}
\end{theorem}
\textbf{Remarks.}
\begin{itemize}
    \item Calculations in~\cite{KS05a} show that if $B_{2r}^{\mathrm{odd}}(n,m)$ is an optimal $B_{2r}^{\mathrm{odd}}$-construction, then $m < \sqrt{2rn}$.
    So it suffices to consider $m$ in the range $\left[0,\sqrt{2r(n-t)}\right]$ for Theorem~\ref{THM:expanded-triangle}. 
    \item In general, one could consider the expanded $K_{\ell+1}$ for $\ell\ge 3$. It seems that the above theorem can be extended to these hypergraphs in some cases. We refer the reader to~\cite{Sido94Tri} and~\cite{KS05a} for more details. 
\end{itemize}

\subsection{Hypergraph books}\label{SUBSEC:books}
Let $F_7$ ($4$-book with $3$-pages) denote the $3$-graph with vertex set $\{1,2,3,4,5,6,7\}$ and edge set 
\begin{align*}
    \left\{1234, 1235, 1236, 1237, 4567\right\}. 
\end{align*}
Let $B_{4}^{\mathrm{even}}(n)$ denote the maximum $B_{4}^{\mathrm{even}}:=\left(2, \{1,1,2,2\}\right)$-construction on $n$ vertices. 
Simply calculations show that $|B_4(n)| \sim \frac{3}{8}\binom{n}{4}$. 

\begin{figure}[htbp]
\centering
\tikzset{every picture/.style={line width=1pt}} 
\begin{tikzpicture}[x=0.75pt,y=0.75pt,yscale=-1,xscale=1,line join=round]

\draw   (492.85,231.61) .. controls (471.76,231.62) and (454.64,195.6) .. (454.62,151.15) .. controls (454.6,106.7) and (471.68,70.66) .. (492.77,70.65) .. controls (513.86,70.64) and (530.98,106.66) .. (531,151.11) .. controls (531.02,195.56) and (513.94,231.6) .. (492.85,231.61) -- cycle ;

\draw [fill=uuuuuu, fill opacity=0.5]   (131,109) .. controls (154,121) and (169,138) .. (136,157) .. controls (165,152) and (156.73,187.54) .. (141,210) .. controls (180,169) and (212,137) .. (256,126) .. controls (226,128) and (169,117) ..  (131,109) -- cycle;

\draw [fill=uuuuuu, fill opacity=0.5]   (131,109) .. controls (154,121) and (169,138) .. (136,157) .. controls (165,152) and (156.73,187.54) .. (141,210) .. controls (178,186) and (212,183) .. (257.37,180.69)  .. controls (221,182) and (170,126) .. (131,109) -- cycle;

\draw [fill=uuuuuu, fill opacity=0.5]   (131,109) .. controls (154,121) and (169,138) .. (136,157) .. controls (165,152) and (156.73,187.54) .. (141,210)  .. controls (183.73,150.54) and (229,91) .. (256,74) .. controls (214,91) and (190,105) .. (131,109);
\draw  [fill=uuuuuu, fill opacity=0.5]  (256,74) .. controls (274.28,83.53) and (268.91,126) .. (256,126) .. controls (274.28,135.53) and (270.27,180.69) .. (257.37,180.69) .. controls (275.65,190.22) and (276.91,234) .. (264,234) .. controls (296.26,221.75) and (281.81,83.53) .. (256,74) -- cycle;
\draw   (387.85,233.61) .. controls (366.76,233.62) and (349.64,197.6) .. (349.62,153.15) .. controls (349.6,108.7) and (366.68,72.66) .. (387.77,72.65) .. controls (408.86,72.64) and (425.98,108.66) .. (426,153.11) .. controls (426.02,197.56) and (408.94,233.6) .. (387.85,233.61) -- cycle ;
\draw   [fill=uuuuuu, fill opacity=0.5]   (375,132) .. controls (414,156) and (452,151) .. (488.35,128.8) .. controls (469.5,142.13) and (473.94,161.01) .. (486.13,173.23)  .. controls (445,155) and (422,160) .. (372.78,176.43)  .. controls (396,160) and (390,145) .. (375,132) -- cycle;
\draw [fill=uuuuuu] (131,109) circle (2pt);
\draw [fill=uuuuuu] (136,157) circle (2pt);
\draw [fill=uuuuuu] (141,210) circle (2pt);
\draw [fill=uuuuuu] (256,126) circle (2pt);
\draw [fill=uuuuuu] (257.37,180.69) circle (2pt);
\draw [fill=uuuuuu] (256,74) circle (2pt);
\draw [fill=uuuuuu] (264,234) circle (2pt);
\draw [fill=uuuuuu] (375,132) circle (1.2pt);
\draw [fill=uuuuuu] (488.35,128.8) circle (1.2pt);
\draw [fill=uuuuuu] (486.13,173.23) circle (1.2pt);
\draw [fill=uuuuuu] (372.78,176.43) circle (1.2pt);

\draw (374.16,240) node [anchor=north west][inner sep=0.75pt]   [align=left] {$V_1$};
\draw (482.77,240) node [anchor=north west][inner sep=0.75pt]   [align=left] {$V_2$};
\draw (116,104) node [anchor=north west][inner sep=0.75pt]   [align=left] {$1$};
\draw (116,148) node [anchor=north west][inner sep=0.75pt]   [align=left] {$2$};
\draw (116,202) node [anchor=north west][inner sep=0.75pt]   [align=left] {$3$};
\draw (250,81.48) node [anchor=north west][inner sep=0.75pt]   [align=left] {$4$};
\draw (250,131.69) node [anchor=north west][inner sep=0.75pt]   [align=left] {$5$};
\draw (250,186) node [anchor=north west][inner sep=0.75pt]   [align=left] {$6$};
\draw (250,228.69) node [anchor=north west][inner sep=0.75pt]   [align=left] {$7$};
\end{tikzpicture}
\caption{The $4$-graph $F_7$ ($4$-book with $3$ pages) and the $4$-graph $B_{4}^{\mathrm{even}}(n)$.}
\label{fig:F7}
\end{figure}
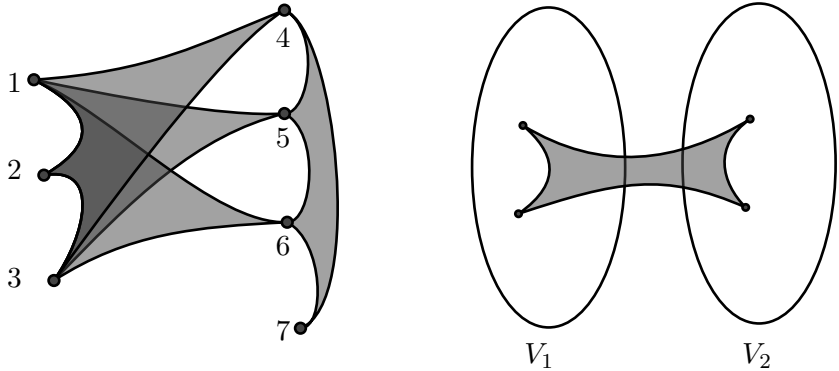

F\"{u}redi, Pikhurko, and Simonovits~\cite{FPS06Book} proved that $\mathrm{EX}(n, F_7) = \{B_4(n)\}$ for all sufficiently large $n$. 
Moreover, they proved that $F_7$ is degree-stable. 
Hence, we obtain the following result. 

\begin{theorem}\label{THM:Hypergraph-Book}
    There exist constants $N_0$ and $c > 0$ such that for all integers $n \ge N_0$ and $t\in [0, c n]$, we have 
    \begin{align*}
        \mathrm{EX}\left(n, (t+1)F_7\right) = \left\{K_{t}^{4} \uproduct B^{\mathrm{even}}_{4}(n-t)\right\}. 
    \end{align*}
\end{theorem}

Let $\mathbb{F}_{4,3}$ denote the $4$-graph with vertex set $\{1,2,3,4,5,6,7\}$ and edge set
\begin{align*}
    \left\{1234, 1235, 1236, 1237, 4567\right\}. 
\end{align*}
Let $B_{4}^{\mathrm{odd}}(n,m)$ denote the $B_4^{\mathrm{odd}}:=\left(2, \{\{1,2,2,2\}, \{1,1,1,2\}\}\right)$-construction on $n$ vertices with one part of size $\lfloor n/2 \rfloor + m$. 
Recall from the previous subsection that $\max_{m}|B_{4}^{\mathrm{odd}}(n,m)| \sim \frac{1}{2}\binom{n}{4}$. 

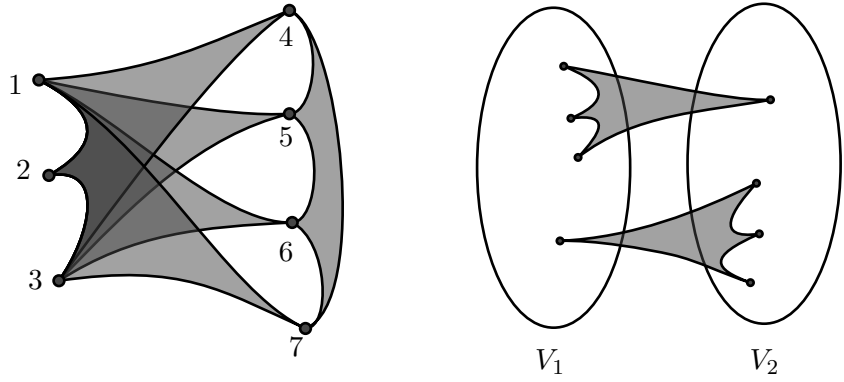
\begin{figure}[htbp]
\centering
\tikzset{every picture/.style={line width=1pt}} 
\begin{tikzpicture}[x=0.75pt,y=0.75pt,yscale=-1,xscale=1,line join=round]

\draw   (493.85,202.65) .. controls (472.76,202.66) and (455.64,166.64) .. (455.62,122.19) .. controls (455.6,77.75) and (472.68,41.71) .. (493.77,41.7) .. controls (514.86,41.69) and (531.98,77.71) .. (532,122.15) .. controls (532.02,166.6) and (514.94,202.64) .. (493.85,202.65) -- cycle ;

\draw   [fill=uuuuuu, fill opacity=0.5]   (132,80.04) .. controls (155,92.04) and (170,109.04) .. (137,128.04)  .. controls (166,123.04) and (157.73,158.58) .. (142,181.04) .. controls (181,140.04) and (213,108.04) .. (257,97.04) .. controls (227,99.04) and (170,88.04) .. (132,80.04) -- cycle;

\draw   [fill=uuuuuu, fill opacity=0.5]   (132,80.04) .. controls (155,92.04) and (170,109.04) .. (137,128.04)  .. controls (166,123.04) and (157.73,158.58) .. (142,181.04) .. controls (179,157.04) and (213,154.04) .. (258.37,151.73) .. controls (222,153.04) and (171,97.04) .. (132,80.04) -- cycle;

\draw   [fill=uuuuuu, fill opacity=0.5]   (132,80.04) .. controls (155,92.04) and (170,109.04) .. (137,128.04)  .. controls (166,123.04) and (157.73,158.58) .. (142,181.04) .. controls (184.73,121.58) and (230,62.04) .. (257,45.04)  .. controls (215,62.04) and (191,76.04)  ..  (132,80.04) -- cycle;

\draw   [fill=uuuuuu, fill opacity=0.5]   (132,80.04) .. controls (155,92.04) and (170,109.04) .. (137,128.04)  .. controls (166,123.04) and (157.73,158.58) .. (142,181.04)  .. controls (204,172) and (223,186) .. (265,205.04) .. controls (224,186) and (187,108) .. (132,80.04) --cycle;

\draw  [fill=uuuuuu, fill opacity=0.5]   (257,45.04) .. controls (275.28,54.57) and (269.91,97.04) .. (257,97.04) .. controls (275.28,106.57) and (271.27,151.73) .. (258.37,151.73) .. controls (276.65,161.26) and (277.91,205.04) .. (265,205.04) .. controls (297.26,192.8) and (282.81,54.57) .. (257,45.04) --cycle;
\draw   (388.85,204.65) .. controls (367.76,204.66) and (350.64,168.64) .. (350.62,124.19) .. controls (350.6,79.75) and (367.68,43.71) .. (388.77,43.7) .. controls (409.86,43.69) and (426.98,79.71) .. (427,124.15) .. controls (427.02,168.6) and (409.94,204.64) .. (388.85,204.65) -- cycle ;

\draw  [fill=uuuuuu, fill opacity=0.5]  (393.85,73.14) .. controls (410.42,79.69) and (421.22,88.97) .. (397.45,99.33) .. controls (418.34,96.6) and (412.32,106.74) .. (401,119) .. controls (429.09,96.63) and (465.31,96) .. (497,90) .. controls (475.4,91.09) and (421.22,77.51) .. (393.85,73.14) -- cycle;
 
\draw  [fill=uuuuuu, fill opacity=0.5]   (487,182) .. controls (468.36,175.05) and (465.1,169.04) .. (491.43,157.45)  .. controls (468.16,160.67) and (477.6,145.51) .. (490,132)  .. controls (460.51,146.61) and (428.47,159.24) .. (392,161) .. controls (408.94,160.37) and (435,165) .. (448,168) .. controls (461,171)  and (473.72,177.88).. (487,182) -- cycle;
\draw [fill=uuuuuu]  (132,80.04) circle (2pt);
\draw [fill=uuuuuu] (137,128.04) circle (2pt);
\draw [fill=uuuuuu] (142,181.04) circle (2pt);
\draw [fill=uuuuuu] (265,205.04) circle (2pt);
\draw [fill=uuuuuu] (257,45.04) circle (2pt);
\draw [fill=uuuuuu] (258.37,151.73) circle (2pt);
\draw [fill=uuuuuu] (257,97.04) circle (2pt);

\draw [fill=uuuuuu]  (393.85,73.14) circle (1.2pt);
\draw [fill=uuuuuu] (397.45,99.33) circle (1.2pt);
\draw [fill=uuuuuu] (401,119) circle (1.2pt);
\draw [fill=uuuuuu] (497,90) circle (1.2pt);
\draw [fill=uuuuuu] (487,182) circle (1.2pt);
\draw [fill=uuuuuu]  (491.43,157.45) circle (1.2pt);
\draw [fill=uuuuuu] (490,132) circle (1.2pt);
\draw [fill=uuuuuu] (392,161) circle (1.2pt);
\draw (378,214) node [anchor=north west][inner sep=0.75pt]   [align=left] {$V_1$};
\draw (485,214) node [anchor=north west][inner sep=0.75pt]   [align=left] {$V_2$};
\draw (115,76) node [anchor=north west][inner sep=0.75pt]   [align=left] {$1$};
\draw (119,119) node [anchor=north west][inner sep=0.75pt]   [align=left] {$2$};
\draw (125,174) node [anchor=north west][inner sep=0.75pt]   [align=left] {$3$};
\draw (250,52) node [anchor=north west][inner sep=0.75pt]   [align=left] {$4$};
\draw (250,102) node [anchor=north west][inner sep=0.75pt]   [align=left] {$5$};
\draw (250,160) node [anchor=north west][inner sep=0.75pt]   [align=left] {$6$};
\draw (255,208) node [anchor=north west][inner sep=0.75pt]   [align=left] {$7$};

\end{tikzpicture}

\caption{The $4$-graph $\mathbb{F}_{4,3}$ and the $4$-graph $B_{4}^{\mathrm{odd}}(n)$.}
\label{fig:F43}
\end{figure}

F\"{u}redi, Mubayi, and Pikhurko~\cite{FMP084Book4page} proved that $\mathrm{EX}(n, \mathbb{F}_{4,3}) \subset \{B_{4}^{\mathrm{odd}}(n,m) \colon m \in [0,n/2]\}$ for large $n$, and moreover,  $\mathbb{F}_{4,3}$ is edge-stable with respect to $B_4^{\mathrm{odd}}$. 
They also showed that edge-stable cannot be replaced by degree-stable (or vertex-extendable). 
However, from Lemma~3.1 in~\cite{FMP084Book4page} one can easily obtain that $\mathbb{F}_{4,3}$ is weakly edge-stable with respect to $B_4^{\mathrm{odd}}$. 
Hence, we obtain the following theorem. 

\begin{theorem}\label{THM:Hypergraph-Book-4-4}
    There exist constants $N_0$ and $c > 0$ such that for all integers $n \ge N_0$ and $t\in [0, c n]$, we have 
    \begin{align*}
        \mathrm{EX}\left(n, (t+1)\mathbb{F}_{4,3}\right) \subset K_{t}^{4} \uproduct \left\{B_{4}^{\mathrm{odd}}(n-t,m) \colon m \in [0, \sqrt{4(n-t)}]\right\}. 
    \end{align*}
\end{theorem}

Let $\mathbb{F}_{3,2}$ denote the $3$-graph with vertex set $\{1,2,3,4,5\}$ and edge set
\begin{align*}
    \{123,124,125,345\}. 
\end{align*}
Recall that $S_{3}(n)$ is the semibipartite $3$-graph on $n$ vertices with the maximum number of edges, i.e. the maximum $S_3:=\left(2, \{1,2,2\}\right)$-construction on $n$ vertices.

\begin{figure}[htbp]
\centering

\tikzset{every picture/.style={line width=1pt}} 

\begin{tikzpicture}[x=0.75pt,y=0.75pt,yscale=-1,xscale=1,line join=round]

\draw   (394.02,181.74) .. controls (377.73,181.66) and (364.64,159.48) .. (364.78,132.2) .. controls (364.91,104.92) and (378.23,82.87) .. (394.52,82.96) .. controls (410.81,83.04) and (423.9,105.22) .. (423.77,132.5) .. controls (423.63,159.78) and (410.31,181.83) .. (394.02,181.74) -- cycle ;
\draw   (505.85,213.61) .. controls (484.76,213.62) and (467.64,177.6) .. (467.62,133.15) .. controls (467.6,88.7) and (484.68,52.66) .. (505.77,52.65) .. controls (526.86,52.64) and (543.98,88.66) .. (544,133.11) .. controls (544.02,177.56) and (526.94,213.6) .. (505.85,213.61) -- cycle ;
\draw   [fill=uuuuuu, fill opacity=0.5]  (386.08,133.02) .. controls (428.2,130.8) and (481.4,129.68) .. (501.35,110.8)  .. controls (482.5,124.13) and (486.94,143.01) .. (499.13,155.23) .. controls (480.29,135.24) and (431.52,137.46) .. (386.08,133.02) --cycle;
\draw   [fill=uuuuuu, fill opacity=0.5]  (158.9,88.86) .. controls (172.88,105.19) and (175.03,132.41) .. (164.27,151.46) .. controls (196.54,109.27) and (220.2,90.22) .. (261.07,63) .. controls (230.95,79.33)  and (195.46,84.78).. (158.9,88.86) -- cycle;

\draw   [fill=uuuuuu, fill opacity=0.5]  (158.9,88.86) .. controls (172.88,105.19) and (175.03,132.41) .. (164.27,151.46) .. controls (197.61,128.33) and (223.43,129.69) .. (265.37,129.69)  .. controls (232.03,126.97) and (194.39,103.83) .. (158.9,88.86) -- cycle;

\draw   [fill=uuuuuu, fill opacity=0.5]  (158.9,88.86) .. controls (172.88,105.19) and (175.03,132.41) .. (164.27,151.46) .. controls (203.64,153.94) and (235.56,179.62) .. (266.44,205.9) .. controls  (252.03,182.83) and (188.2,113.28) .. (158.9,88.86) -- cycle;
\draw  [fill={rgb:red,1;green,0;blue,0}, fill opacity=0.5]  (261.07,63) .. controls (279.35,72.53) and (278.27,129.69) .. (265.37,129.69) .. controls (283.65,139.22) and (279.35,205.9) .. (266.44,205.9) .. controls (298.71,193.65) and (286.88,72.53) .. (261.07,63) -- cycle;
\draw [fill=uuuuuu] (261.07,63) circle (2pt);
\draw [fill=uuuuuu] (265.37,129.69) circle (2pt);
\draw [fill=uuuuuu]  (266.44,205.9) circle (2pt);
\draw [fill=uuuuuu] (158.9,88.86) circle (2pt);
\draw [fill=uuuuuu] (164.27,151.46) circle (2pt);
\draw [fill=uuuuuu] (386.08,133.02) circle (1.2pt);
\draw [fill=uuuuuu] (501.35,110.8) circle (1.2pt);
\draw [fill=uuuuuu] (499.13,155.23)  circle (1.2pt);
\draw (387.16,191.96) node [anchor=north west][inner sep=0.75pt]   [align=left] {$V_1$};
\draw (495.77,220.84) node [anchor=north west][inner sep=0.75pt]   [align=left] {$V_2$};
\draw (145,80) node [anchor=north west][inner sep=0.75pt]   [align=left] {$1$};
\draw (145,144) node [anchor=north west][inner sep=0.75pt]   [align=left] {$2$};
\draw (256,68.79) node [anchor=north west][inner sep=0.75pt]   [align=left] {$3$};
\draw (256,135.48) node [anchor=north west][inner sep=0.75pt]   [align=left] {$4$};
\draw (256,211.69) node [anchor=north west][inner sep=0.75pt]   [align=left] {$5$};
\end{tikzpicture}
\caption{The $3$-graph $\mathbb{F}_{3,2}$ and the semibipartite $3$-graph $S_3(n)$.}
\label{fig:F32}
\end{figure}
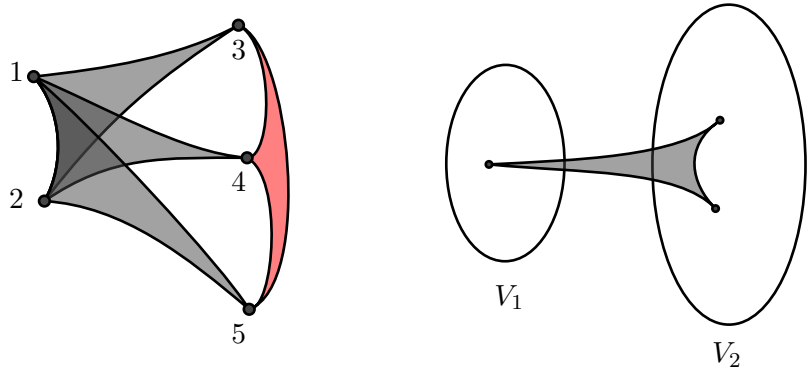

F\"{u}redi, Pikhurko, and Simonovits~\cite{FPS053Book3page} proved that $\mathrm{EX}(n, \mathbb{F}_{3,2}) = \{S_3(n)\}$ for all sufficiently large $n$. 
A construction in their paper ({\cite[Construction~1.2]{FPS053Book3page}}) shows that $\mathbb{F}_{3,2}$ is not vertex-extendable with respect $S_3$. 
But we will present a short proof in Section~\ref{SEC:proof-3book} which shows that $\mathbb{F}_{3,2}$ is weakly vertex-extendable with respect to $S_3$. 
Hence, we obtain the following result. 

\begin{theorem}\label{THM:Hypergraph-Book-3}
    There exist constants $N_0$ and $c > 0$ such that for all integers $n \ge N_0$ and $t\in [0, c n]$, we have 
    \begin{align*}
        \mathrm{EX}\left(n, (t+1)\mathbb{F}_{3,2}\right) = \left\{K_{t}^{r} \uproduct S_{3}(n-t)\right\}. 
    \end{align*}
\end{theorem}

\section{Proofs of Theorems~\ref{THM:Main-simple} and~\ref{THM:main-rainbow}}\label{SEC:Proof-Main}
In this section, we prove Theorems~\ref{THM:Main-simple} and~\ref{THM:main-rainbow}.
In fact, we will prove the following more general (but also more technical) version. 

\begin{theorem}\label{THEOREM:main-techincal}
Let $m \ge r \ge 2$ be integers and $F$ be a nondegenerate $r$-graph on $m$ vertices. 
Let $f \colon \mathbb{N}\to \mathbb{R}$ be a nondecreasing function. 
Suppose that for all sufficiently large $n\in \mathbb{N} \colon$ 
\begin{enumerate}[label=(\alph*)]
\item\label{lab:smooth} 
   $\mathrm{ex}(n, F)$ is $\frac{1-\pi(F)}{8m}\binom{n}{r-1}$-smooth, and   
\item\label{lab:bounded} 
 $F$ is $\left(f(n), \frac{1-\pi(F)}{4m} \binom{n}{r-1}\right)$-bounded. 
\end{enumerate}
Then there exists $N_0$ such that the following statements hold for all integers $n, t \in \mathbb{N}$ with 
\begin{align*}
    n\ge N_0, \quad t \le \frac{1-\pi(F)}{64rm^2}n, \quad\text{and}\quad
     2emt\binom{n-2mt}{r-2} \le f(n-2mt). 
\end{align*}
\begin{enumerate}[label=(\roman*)]
\item\label{lab:tech-rainbow} 
   If a collection $\{\mathcal{H}_1, \ldots, \mathcal{H}_{t+1}\}$ of $n$-vertex $r$-graphs on the same vertex set satisfies 
\begin{align*}
    |\mathcal{H}_i| > \binom{n}{r} - \binom{n-t}{r} + \mathrm{ex}(n-t, F) \quad\text{for all } i\in [t+1], 
\end{align*}
then $\{\mathcal{H}_1, \ldots, \mathcal{H}_{t+1}\}$ contains a rainbow $F$-matching.   
\item\label{lab:tech-Turan} 
  We have  $\mathrm{EX}(n, (t+1)F) = K_{t}^r\uproduct \mathrm{EX}(n-t,F)$. 
\end{enumerate}

\end{theorem}
\subsection{Preparations}\label{SUBSEC:Proof-Prep}
First, recall the following result due to Katona, Nemetz, and Simonovits~\cite{KNS64}

\begin{proposition}[Katona--Nemetz--Simonovits~\cite{KNS64}]\label{PROP:KNS}
Fix an $r$-graph $F$. 
The ration $\frac{\mathrm{ex}(n,F)}{\binom{n}{r}}$ is nonincreasing in $n$. 
In particular, $\mathrm{ex}(n,F) \ge \pi(F)\binom{n}{r}$ for all $n\in \mathbb{N}$, and
\begin{align*}
    \pi(F) 
    \le \frac{\mathrm{ex}(v(F),F)}{\binom{v(F)}{r}} 
    \le \frac{\binom{v(F)}{r}-1}{\binom{v(F)}{r}}<1.
\end{align*}
\end{proposition}

Next, we prove two simple inequalities concerning binomials. 

\begin{lemma}\label{LEMMA:Binomal-Inequ}
    Suppose that $m \le n/r -1$. Then 
    \begin{align}\label{equ:binomal-ineq}
    \binom{n}{r}-\binom{n-m}{r}
    =\sum_{i=1}^{r}\binom{m}{i}\binom{n-m}{r-i}
    \le 2m\binom{n-m}{r-1}.
    \end{align}
\end{lemma}
\begin{proof}
    For every $i\in [2,r]$ we have 
    \begin{align*}
        \frac{\binom{m}{i} \binom{n-m}{r-i}}{\binom{m}{i-1}\binom{n-m}{r-i+1}} 
        = \frac{m-i+1}{i}\frac{r-i+1}{n-m-r+i}
        \le \frac{(r-1)m}{2(n-m-r)}
        \le \frac{1}{2}, 
    \end{align*}
    where the last inequality follows from the assumption that $m \le n/r -1$.
    Therefore, 
    \begin{align*}
        \sum_{i=1}^{r}\binom{m}{i}\binom{n-m}{r-i}
         \le \sum_{i=1}^{r}\left(\frac{1}{2}\right)^{i-1} m \binom{n-m}{r-1}
         \le 2 m \binom{n-m}{r-1}. 
    \end{align*}
\end{proof}

\begin{lemma}\label{LEMMA:binom-inequ-b}
Suppose that integers $n,b,r \ge 1$ satisfy $b \le \frac{n-r}{r+1}$. Then 
    \begin{align*}
        \binom{n}{r} \le e \binom{n-b}{r}. 
    \end{align*}
\end{lemma}
\begin{proof}
    For every $i \in [b]$ it follows from $b \le \frac{n-r}{r+1}$ that $\frac{n-i}{n-i-r} = 1+ \frac{r}{n-i-r} \le 1+ \frac{r}{n-b-r} \le 1+ \frac{1}{b}$.
    Therefore, 
    \begin{align*}
        \binom{n}{r} 
        = \prod_{i=0}^{b-1}\frac{n-i}{n-i-r} \binom{n-b}{r}
        \le \left(1+ \frac{1}{b}\right)^b \binom{n-b}{r}
        \le e \binom{n-b}{r}. 
    \end{align*}
\end{proof}

The following lemma says that $d(n,F)$ is smooth for every $F$. 

\begin{lemma}\label{LEMMA:smooth-degree}
    Let $F$ be an $r$-graph. 
    For every $n$ and $m \le  n/r -1$ we have 
    \begin{align*}
        \left| d(n,F) - d(n-m, F) \right| \le 4m \binom{n-m}{r-2}.
    \end{align*}
\end{lemma}
\begin{proof}
    It follows from Proposition~\ref{PROP:KNS} that $\mathrm{ex}(n,F)/\binom{n}{r} \le \mathrm{ex}(n-m,F)/\binom{n-m}{r}$. 
    Therefore, 
    \begin{align*}
        \mathrm{ex}(n,F) - \mathrm{ex}(n-m,F)
        & \le \frac{\binom{n}{r}}{\binom{n-m}{r}}\mathrm{ex}(n-m,F) - \mathrm{ex}(n-m,F) \\
        & = \frac{\binom{n}{r}-\binom{n-m}{r}}{\binom{n-m}{r}}\mathrm{ex}(n-m,F) \\
        & \overset{\text{Lemma~} \ref{LEMMA:Binomal-Inequ}}\le \frac{2m \binom{n-m}{r-1}}{\binom{n-m}{r}}\mathrm{ex}(n-m,F) 
         = \frac{2mr}{n-m-r+1}\mathrm{ex}(n-m,F). 
    \end{align*}
    Consequently, 
    \begin{align*}
        \left| d(n,F) - d(n-m, F) \right|
        & =\left| \frac{r\cdot \mathrm{ex}(n,F)}{n} - \frac{r\cdot \mathrm{ex}(n-m,F)}{n-m} \right| \\
        & = \left| \frac{r}{n}\left(\mathrm{ex}(n,F)-\mathrm{ex}(n-m,F)\right) - \frac{rm}{n(n-m)}\mathrm{ex}(n-m,F) \right|\\
        & \le \max\left\{ \frac{2mr^2}{n(n-m-r+1)},\frac{rm}{n(n-m)} \right\} \cdot \mathrm{ex}(n-m,F) \\
        & \le \frac{2mr^2}{n(n-m-r+1)} \binom{n-m}{r} 
        \le 4m \binom{n-m}{r-2}. 
    \end{align*}
    This completes the proof of Lemma~\ref{LEMMA:smooth-degree}. 
\end{proof}

The following lemma deals with a simple case of Theorem~\ref{THEOREM:main-techincal} in which the maximum degree of every $r$-graph $\mathcal{H}_i$ is bounded away from $\binom{n-1}{r-1}$.  

\begin{lemma}\label{LEMMA:Case1-small-maxdeg-rainbow}
    Let $F$ be a nondegenerate $r$-graph with $m$ vertices.
    Suppose that $\mathrm{ex}(n,F)$ is $g$-smooth with $g(n) \le \frac{1-\pi(F)}{8m}\binom{n}{r-1}$ for all sufficiently large $n$. 
    Then there exists $N_1$ such that the following holds for all integers $n, t\in \mathbb{N}$ with $n\ge N_1$ and $t \le \frac{1-\pi(F)}{64rm^2} n$. 
    
    Suppose that $\left\{\mathcal{H}_1, \ldots, \mathcal{H}_{t+1}\right\}$ is a collection of $n$-vertex $r$-graphs on the same vertex set $V$ such that 
    \begin{align}
        |\mathcal{H}_i| \ge \mathrm{ex}(n-t,F) +t\binom{n-t}{r-1} \quad\text{and}\quad
        \Delta(\mathcal{H}_i) \le d(n-t,F)+ \frac{1-\pi(F)}{2m} \binom{n-t}{r-1} \notag
    \end{align}
    hold for all $i\in [t+1]$. 
    Then $\left\{\mathcal{H}_1, \ldots, \mathcal{H}_{t+1}\right\}$ contains a rainbow $F$-matching. 
\end{lemma} 
\begin{proof}
    Given an integer $k \le t+1$, we say a collection $\mathcal{C} = \{S_1, \ldots, S_{k}\}$ of pairwise disjoint $m$-subsets of $V$ is \textbf{$F$-rainbow} if there exists an injection $f \colon [k]\to [t+1]$ such that $F \subset \mathcal{H}_{f(i)}[S_i]$ for all $i\in [k]$. 

    Fix a maximal collection $\mathcal{C} = \{S_1, \ldots, S_{k}\}$ of pairwise disjoint $m$-subsets of $V$ that is $F$-rainbow. 
    If $k = t+1$, then we are done. So we may assume that $k \le t$. 
    Without loss of generality, we may assume that $F \subset \mathcal{H}_i[S_i]$ for all $i\in [k]$ (i.e. $f$ is the identity map).
    Let $B = \bigcup_{i=1}^{k}S_i$ and let $b = |B| = mk$. 

    Let us count the number of edges in $\mathcal{H}_{k+1}$.
    Observe that every copy of $F$ in $\mathcal{H}_{k+1}$ must contain a vertex from $B$, 
    since otherwise, it would contradict the maximality of $\mathcal{C}$. 
    Therefore, the induced subgraph of $\mathcal{H}_{k+1}$ on $V_0:= V\setminus B$ is $F$-free. 
    Hence, by the maximum degree assumption, we obtain 
\begin{align*}
    |\mathcal{H}_{k+1}|
    & \le |\mathcal{H}_{k+1}[V_0]| 
        + b \left(d(n-t, F) + \frac{1-\pi(F)}{2m} \binom{n-t}{r-1} \right) \\
    & \le \mathrm{ex}(n-b, F) + b \left(d(n-t, F) + \frac{1-\pi(F)}{2m} \binom{n-t}{r-1} \right) \\
    & = \mathrm{ex}(n-t, F) + t \binom{n-t}{r-1} 
    - \left(\Delta_1 + \Delta_2\right), 
\end{align*}
where 
\begin{align*}
    \Delta_1 & := \mathrm{ex}(n-t, F)-\mathrm{ex}(n-b, F) - (b-t)d(n-t, F), \\
    \Delta_2 & := t \left(\binom{n-t}{r-1} - d(n-t, F)\right) - b \frac{1-\pi(F)}{2m} \binom{n-t}{r-1}.  
\end{align*}
Next, we will prove that $\Delta_1 + \Delta_2 >0$, which implies that $|\mathcal{H}_{k+1}|
    < \mathrm{ex}(n-t, F) + t \binom{n-t}{r-1}$ contradicting our assumption. 
    
Since $n-t \ge N_1/2$ is sufficiently large and $\lim_{n\to \infty}\mathrm{ex}(n-t,F)/\binom{n-t}{r} = \pi(F)$, we have 
$\mathrm{ex}(n-t,F) \le \left(\pi(F)+ \frac{1-\pi(F)}{5}\right)\binom{n-t}{r}$, and hence, 
\begin{align*}
    d(n-t, F) = \frac{r\cdot \mathrm{ex}(n-t,F)}{n-t} \le \left(\pi(F)+ \frac{1-\pi(F)}{5}\right)\binom{n-t}{r-1}.  
\end{align*}
Therefore, 
\begin{align*}
    \Delta_2
    & \ge t \left(1 -  \left(\pi(F)+ \frac{1-\pi(F)}{5}\right)\right) \binom{n-t}{r-1}
        - mt \frac{1-\pi(F)}{2m} \binom{n-t}{r-1} \\
    & \ge \frac{1-\pi(F)}{4}\binom{n-t}{r-1} t. 
\end{align*}
On the other hand, 
by Lemma~\ref{LEMMA:smooth-degree}, we have 
$$d(n-t, F) \le d(n-b,F) + 4(b-t)\binom{n-b}{r-2} \le d(n-b, F) + 4mt\binom{n-t}{r-2}.$$
Therefore, it follows from the Smoothness assumption and $g$ is nondecreasing that  
\begin{align*}
    \Delta_1 
     & = \sum_{i=1}^{b-t}\left(\mathrm{ex}(n-b+i, F) - \mathrm{ex}(n-b+i-1, F)\right) - (b-t)d(n-t,F) \\
     \overset{\text{Smoothness}}&{\ge} \sum_{i=0}^{b-t-1}\left(d(n-b+i, F) - g(n-b+i+1)\right) - (b-t)d(n-t,F) \\
    \overset{\text{Nondecreasing}}&{\ge} \sum_{i=0}^{b-t-1}\left(d(n-b+i, F) - d(n-t,F)\right) - (b-t)g(n-t) \\
    \overset{\text{Lemma~\ref{LEMMA:smooth-degree}}}&{\ge} 
                -\sum_{i=0}^{b-t-1}4(b-t-i)\binom{n-b+i}{r-2} - (b-t)g(n-t) \\
    & \ge - 4m^2t^2 \binom{n-t-1}{r-2} - mt \cdot g(n-t)
    = - \frac{4(r-1)m^2t^2}{n-t} \binom{n-t}{r-1} - mt \cdot  g(n-t). 
\end{align*}
Since $t \le \frac{1-\pi(F)}{64rm^2} n$, we obtain 
$\frac{4(r-1)m^2t^2}{n-t} < \frac{1-\pi(F)}{8} t$.  
Together with $g(n-t) \le \frac{1-\pi(F)}{8m}\binom{n-t}{r-1}$, we obatin 
\begin{align*}
    \Delta_1  
    > -\left(\frac{1-\pi(F)}{8} t + m t \frac{1-\pi(F)}{8m}\right) \binom{n-t}{t-1}
    = - \frac{1-\pi(F)}{4} t \binom{n-t}{r-1}. 
\end{align*}
Therefore, $\Delta_1 + \Delta_2 > 0$. 
This finishes the proof of Lemma~\ref{LEMMA:Case1-small-maxdeg-rainbow}. 
\end{proof}

\subsection{Proof of Theorem~\ref{THEOREM:main-techincal}}\label{SUBSEC:proof-main-tech}
We prove Theorem~\ref{THEOREM:main-techincal} in this section. 
Let us prove Part~\ref{lab:tech-rainbow} first. 

\begin{proof}[Proof of Theorem~\ref{THEOREM:main-techincal}~\ref{lab:tech-rainbow}]
Fix a sufficiently large constant $N_0$ and suppose that $n \ge N_0$. 
Let $k \le t+1$. 
We say a collection $L:= \{v_1, \ldots, v_k\}$ of vertices in $V$ is \textbf{heavy-rainbow} if there exists an injection $f \colon [k] \to [t+1]$ such that 
\begin{align*}
    d_{\mathcal{H}_{f(i)}}(v_i) 
    \ge d(n-t, F) + \frac{1-\pi(F)}{2m} \binom{n-t}{r-1} \quad\text{for all } i\in [k]. 
\end{align*}
Fix a maximal collection $L:= \{v_1, \ldots, v_k\}$ of vertices that is heavy-rainbow.
Without loss of generality, we may assume that $f$ (defined above) is the identity map. 

Let $V_0 = V \setminus L$ and $\mathcal{H}'_{j} = \mathcal{H}_{j}[V_0]$ for all $j\in [k+1, t+1]$. 
For every $j\in [k+1, t+1]$
observe that there are at most $\binom{n}{r}-\binom{n-k}{r}$ edges in $\mathcal{H}_j$ that have nonempty intersection with $L$. Hence, 
\begin{align*}
    |\mathcal{H}'_{j}|
    & \ge |\mathcal{H}_j| - \left(\binom{n}{r}-\binom{n-k}{r}\right) \\
    & \ge \mathrm{ex}(n-t, F) + \binom{n}{r}-\binom{n-t}{r} - \left(\binom{n}{r}-\binom{n-k}{r}\right) \\
    & = \mathrm{ex}((n-k)-(t-k), F) + \binom{n-k}{r} - \binom{(n-k)-(t-k)}{r}. 
\end{align*}
On the other hand, it follows from the maximality of $L$ that 
\begin{align*}
    \Delta(\mathcal{H}'_{j})
    \le \Delta(\mathcal{H}_j)
    & \le d(n-t, F) + \frac{1-\pi(F)}{2m} \binom{n-t}{r-1} \\
    & = d((n-k)-(t-k), F) + \frac{1-\pi(F)}{2m} \binom{(n-k)-(t-k)}{r-1} 
\end{align*}
holds for all $j\in [k+1, t+1]$. 
By assumption, $\frac{t-k}{n-k} \le \frac{t}{n} \le \frac{1-\pi(F)}{64rm^2}$ and $n-k \ge n/2$ is sufficiently large,
so it follows from Lemma~\ref{LEMMA:Case1-small-maxdeg-rainbow} that 
there exists a collection $\mathcal{C} = \{S_{k+1}, \ldots, S_{t+1}\}$ of pairwise disjoint $m$-subsets of $V_0$ such that $F \subset \mathcal{H}'_{j}[S_j]$ for all $j\in [k+1, t+1]$. 

Next we will find a collection of rainbow copies of $F$ from $\{\mathcal{H}_1, \ldots, \mathcal{H}_k\}$. 

\begin{claim}\label{CLAIM:F-avoid-Bi}
For every $i\in [k]$ and for every set $B_i \subset V \setminus \{v_i\}$ of size at most $2mt$ there exists a copy of $F$ in $\mathcal{H}_i[V\setminus B_i]$. 
\end{claim}
\begin{proof}
    Fix $i \in [k]$ and fix a set $B_i \subset V \setminus \{v_i\}$ of size at most $2mt$. We may assume that $|B_i| = 2mt$. 
    Let $V_i = V\setminus B_i$ and  $n_i = |V_i| = n-2mt$. 
    Let $\mathcal{H}_i' = \mathcal{H}_i[V_i]$. 
    Since the number of edges in $\mathcal{H}_i$ containing $v_i$ that have nonempty intersection with $B_i$ is at most $2mt \binom{n-1}{r-2}$, we have 
    \begin{align}\label{equ:Main-max-deg}
        d_{\mathcal{H}_i'}(v_i)
        & \ge d(n-t, F) + \frac{1-\pi(F)}{2m} \binom{n-t}{r-1} - 2mt \binom{n-1}{r-2}  \notag\\
        \overset{\text{Lemma~\ref{LEMMA:smooth-degree}}}&{\ge} 
                d(n-2mt, F) - 2mt\binom{n-2mt}{r-2} + \frac{1-\pi(F)}{2m} \binom{n-t}{r-1} - 2mt \binom{n-1}{r-2} \notag \\
        & > d(n-2mt, F) + \frac{1-\pi(F)}{4m} \binom{n-2mt}{r-1}, 
    \end{align}
    where the last inequality holds because $t \le \frac{1-\pi(F)}{64rm^2}n$ and $n$ is sufficiently large. 
    
    Similarly, we have 
    \begin{align*}
        |\mathcal{H}_i'| 
        \ge |\mathcal{H}_i| - 2mt\binom{n-1}{r-1}
        & > \mathrm{ex}(n-t,F)+ \binom{n}{r}-\binom{n-t}{r}-2mt\binom{n-1}{r-1} \\
         \overset{\text{Lemma~\ref{LEMMA:binom-inequ-b}}}&{\ge}  \mathrm{ex}(n-2mt,F) - 2emt\binom{n-2mt}{r-1}. 
    \end{align*}
    which, by the assumption $f(n-2mt) \ge 2emt\binom{n-2mt}{r-2}$, implies that 
    \begin{align}\label{equ:Main-min-deg}
    d(\mathcal{H}_{i'}) = 
        \frac{r\cdot |\mathcal{H}_i'|}{n-2mt}
        & \ge d(n-2mt,F) - 2emt\binom{n-2mt-1}{r-2} \notag \\
        & > d(n-2mt,F) - f(n-2mt). 
    \end{align}
    It follows from (\ref{equ:Main-max-deg}), (\ref{equ:Main-min-deg}), and the Boundedness assumption that $F \subset \mathcal{H}_{i}'$. 
\end{proof}

Let $B = L \cup S_{k+1}\cup \cdots \cup S_{t+1}$. 
Now we can repeatedly apply Claim~\ref{CLAIM:F-avoid-Bi} to find a collection of  rainbow copies of $F$ as follows. 
First, we let $B_1 = B\setminus \{v_1\}$. 
Since $|B_1| = k-1+m(t+1-k) \le 2mt$, Claim~\ref{CLAIM:F-avoid-Bi} applied to $v_1$, $B_1$, and $\mathcal{H}_1$ yields an $m$-set $S_1 \subset V\setminus B_1$ such that $F \subset \mathcal{H}_1[S_1]$. 
Suppose that we have define $S_1, \ldots, S_i$ for some $i\in [k-1]$ such that $F \subset \mathcal{H}_j[S_j]$ holds for all $j\le i$. 
Then let $B_{i+1} = (B\cup S_1 \cup \cdots \cup S_i) \setminus \{v_{i+1}\}$. 
Since $|B_{i+1}| = k-1+m(t+1-k) + im \le 2 mt$, Claim~\ref{CLAIM:F-avoid-Bi} applied to $v_{i+1}$, $B_{i+1}$, and $\mathcal{H}_{i+1}$ yields an $m$-set $S_{i+1} \subset V\setminus B_{i+1}$ such that $F \subset \mathcal{H}_{i+1}[S_{i+1}]$.
At the end of this process, we obtain a collection $\{S_1, \ldots, S_k\}$ of pairwise disjoint sets such that $F \subset \mathcal{H}_{i}[S_{i}]$ holds for all $i\in [k]$.
Since $S_i \cap S_j = \emptyset$ for all $i\in [k]$ and $j\in [k+1, t+1]$, 
the set $\{S_1, \ldots, S_{t+1}\}$ yields a rainbow $F$-matching. 
\end{proof}

Before proving Part~\ref{lab:tech-Turan} of Theorem~\ref{THEOREM:main-techincal}, we need the simple corollary of Lemma~\ref{LEMMA:Case1-small-maxdeg-rainbow}. 

\begin{lemma}\label{LEMMA:Case1-small-maxdeg}
    Let $F$ be a nondegenerate $r$-graph with $m$ vertices.
    Suppose that $\mathrm{ex}(n,F)$ is $g$-smooth with $g(n) \le \frac{1-\pi(F)}{8m}\binom{n}{r-1}$ for all sufficiently large $n$. 
    Then there exists $N_1$ such that the following holds for all integers $n, t\in \mathbb{N}$ with $n\ge N_1$ and $t \le \frac{1-\pi(F)}{64rm^2} n$. 

    Suppose that $\mathcal{H}$ is an $n$-vertex $r$-graphs with
    \begin{align}
        \Delta(\mathcal{H}) \le d(n-t,F)+ \frac{1-\pi(F)}{2m} \binom{n-t}{r-1}
        \quad\text{and}\quad \nu(F, \mathcal{H})< t+1. \notag
    \end{align}
    Then 
    \begin{align*}
        |\mathcal{H}| < \mathrm{ex}(n-t,F) + t \binom{n-t}{r-1}. 
    \end{align*}
\end{lemma}

Now we are ready to prove  Part~\ref{lab:tech-Turan}. 

\begin{proof}[Proof of Theorem~\ref{THEOREM:main-techincal}~\ref{lab:tech-Turan}] 
    Let $\mathcal{H}$ be an $n$-vertex $r$-graph with $\mathrm{ex}(n, (t+1)F)$ edges and $\nu(F, \mathcal{H}) < t+1$. 
    Note that Theorem~\ref{THEOREM:main-techincal}~\ref{lab:tech-rainbow} already implies that 
    $\mathrm{ex}(n, (t+1)F) \le \binom{n}{r}-\binom{n-t}{r}+\mathrm{ex}(n-t, F)$.
    So, it suffices to show that $\mathcal{H}$ is isomorphic to $K_{t}^r \uproduct \mathcal{G}$ for some $\mathcal{G} \in \mathrm{EX}(n-t,F)$. 

    Let $V = V(\mathcal{H})$ and define 
    \begin{align*}
        L := \left\{v\in V\colon d_{\mathcal{H}}(v) \ge d(n-t, F) + \frac{1-\pi(F)}{2m} \binom{n-t}{r-1}\right\}. 
    \end{align*}

    A similar argument as in the proof of Claim~\ref{CLAIM:F-avoid-Bi} yields the following claim. 
    \begin{claim}\label{CLAIM:F-avoid-Bi-B}
        For every $v\in L$ and for every set $B\subset V\setminus \{v\}$ of size at most $2mt$ there exists a copy of $F$ in $\mathcal{H}[V\setminus B]$. 
    \end{claim}

    Let $\ell=|L|$. We have the following claim for $\ell$. 
    
    \begin{claim}\label{CLAIM:ell<t+1}
        We have $\ell \le t$. 
    \end{claim}
    \begin{proof}
        Suppose to the contrary that $\ell \ge t+1$. By taking a subset of $L$ if necessary, we may assume that $\ell = t+1$. 
        Let us assume that $L = \{v_1, \ldots, v_{t+1}\}$. 
        We will repeatedly apply Claim~\ref{CLAIM:F-avoid-Bi-B} to find a collection $\{S_1, \ldots, S_{t+1}\}$ of pairwise disjoint $m$-sets such that $F\subset \mathcal{H}[S_i]$ for all $i\in [t+1]$ as follows. 
        
        Let $B_1 = L\setminus \{v_1\}$. Since $|B_1| \le 2mt$,  it follows from Claim~\ref{CLAIM:F-avoid-Bi-B} that there exists a set $S_1 \subset V\setminus B$ such that $F\subset \mathcal{H}[S_1]$. 
        Now suppose that we have found pairwise disjoint $m$-sets $S_1, \ldots, S_i$ for some $i \le t$. 
        Let $B_{i+1} = \left(L\cup S_1 \cup \cdots \cup S_i\right) \setminus \{v_i\}$. 
        It is clear that $|B_{i+1}| \le 2mt$. So it follows from Claim~\ref{CLAIM:F-avoid-Bi-B} that there exists a set $S_{i+1} \subset V\setminus B$ such that $F\subset \mathcal{H}[S_{i+1}]$. Repeat this process for $t+1$ times, we find  the collection $\{S_1, \ldots, S_{t+1}\}$ that satisfies the assertion. However, this contradicts the assumption that $\nu(F, \mathcal{H}) < t+1$. 
    \end{proof}

    Let $V_0 = V\setminus L$ and $\mathcal{H}_0 = \mathcal{H}[V_0]$. 
    The following claim follows from 
    a similar argument as in the last paragraph of the proof of Theorem~\ref{THEOREM:main-techincal}. 
    \begin{claim}\label{CLAIM:v(F,H_0)}
        We have $\nu(F, \mathcal{H}_0) < t-\ell+1$. 
    \end{claim}

    If $\ell=t$, then Claim~\ref{CLAIM:v(F,H_0)} implies that $\mathcal{H}_0$ is $F$-free.
    Therefore, it follows from 
    \begin{align*}
        |\mathcal{H}_0| 
        \ge |\mathcal{H}| - \left(\binom{n}{r}-\binom{n-t}{r}\right)
        = \mathrm{ex}(n-t, F)
    \end{align*}
    that $\mathcal{H}_0 \in \mathrm{EX}(n-t, F)$ and $d(v) = \binom{n-1}{r-1}$ for all $v\in L$,
    which implies that $\mathcal{H} = K_{t}^r\uproduct \mathcal{G}$ for some $\mathcal{G} \in \mathrm{EX}(n-t, F)$. 

    If $\ell\le t-1$, then it follows from $\Delta(\mathcal{H}_0) \le d(n-t, F) + \frac{1-\pi(F)}{2m} \binom{n-t}{r-1}$, $\nu(F, \mathcal{H}_0) < t-\ell+1$, and Lemma~\ref{LEMMA:Case1-small-maxdeg} that 
    \begin{align*}
        |\mathcal{H}_0| < \mathrm{ex}(n-t, F) + (t-\ell)\binom{n-t}{r-1}. 
    \end{align*}
    Consequently, 
    \begin{align*}
    |\mathcal{H}| 
    \le  |\mathcal{H}_0| + \binom{n}{r}-\binom{n-\ell}{r} 
    & < \mathrm{ex}(n-t, F) + (t-\ell)\binom{n-t}{r-1} + \binom{n}{r}-\binom{n-\ell}{r} \\
    & \le \mathrm{ex}(n-t, F) + \binom{n}{r}-\binom{n-t}{r}, 
    \end{align*}
    a contradiction. 
\end{proof}

\section{Proofs of Theorems~\ref{THM:Turan-pair-smooth} and~\ref{THM:vertex-extend-bounded}}\label{SEC:Proof-applications}
In this section, we prove Theorems~\ref{THM:Turan-pair-smooth} and~\ref{THM:vertex-extend-bounded}. 
Before that, let us introduce some definitions and prove some preliminary results. 
\subsection{Preliminaries}\label{SUBSEC:Pattern-prelim}
%
%
The following fact concerning $\delta(n,F)$ for all hypergraphs $F$. 

\begin{fact}\label{FACT:min-degree-extremal}
Let $F$ be an $r$-graph and $n\ge 1$ be an integer. 
Then every maximum $n$-vertex $F$-free $r$-graph $\mathcal{H}$ satisfies $\delta(\mathcal{H}) \ge \delta(n,F)$. In particular, $d(n,F) \ge \delta(n,F)$. 
\end{fact}
\begin{proof}
    Let $v\in V(\mathcal{H})$ be a vertex with minimum degree and let $\mathcal{H}'$ be the induced subgraph of $\mathcal{H}$ on $V(\mathcal{H})\setminus \{v\}$.
    Since $\mathcal{H}'$ is an $(n-1)$-vertex $F$-free $r$-graph, we have 
    $|\mathcal{H}'| \le \mathrm{ex}(n-1, F)$. 
    On the other hand, since $\mathcal{H}$ is a maximum $n$-vertex $F$-free $r$-graph, we have $\mathrm{ex}(n,F) = |\mathcal{H}|$. 
    Therefore, 
    \begin{align*}
        \delta(n,F)
        = \mathrm{ex}(n,F) - \mathrm{ex}(n-1,F)
        \le |\mathcal{H}| - |\mathcal{H}'|
        = d_{\mathcal{H}}(v)
        = \delta(\mathcal{H}), 
    \end{align*}
    which proves Fact~\ref{FACT:min-degree-extremal}. 
\end{proof}

For Tur\'{a}n pairs $(F,P)$ we have the following fact which provides a lower bound for $\delta(n,F)$. 

\begin{fact}\label{FACT:(F,P)-Turan-Pair}
    Suppose that $(F,P)$ is a Tur\'{a}n pair and $\mathcal{H}$ is a maximum $F$-free $r$-graph on  $n-1$ vertices. Then $\delta(n,F) \ge \Delta(\mathcal{H})$. 
    In particular, $\delta(n,F) \ge d(n-1,F)$. 
\end{fact}
\begin{proof}
    First, notice that $|\mathcal{H}| = \mathrm{ex}(n-1, F)$. 
    On the other hand, it follows from the definition of Tur\'{a}n pair that $\mathcal{H}$ is an $(n-1)$-vertex $P$-construction. 
    Let $\tilde{\mathcal{H}}$ be an $n$-vertex $P$-construction obtained from $\mathcal{H}$ by duplicating a vertex $v\in V(\mathcal{H})$ with maximum degree. 
    In other words, $\tilde{\mathcal{H}}$ is obtained from $\mathcal{H}$ by adding a new vertex $u$ and adding all edges in $\left\{\{u\} \cup S \colon S \in L_{\mathcal{H}}(v)\right\}$. 
    It is clear that $\tilde{\mathcal{H}}$ is an $n$-vertex $P$-construction, and hence, $\tilde{\mathcal{H}}$ is $F$-free.  
    So $|\tilde{\mathcal{H}}| \le \mathrm{ex}(n,F)$. 
    It follows that 
    \begin{align*}
        \delta(n,F) = \mathrm{ex}(n,F) - \mathrm{ex}(n-1,F)
        \ge |\tilde{\mathcal{H}}| - |\mathcal{H}| 
        = d_{\mathcal{H}}(v)
        = \Delta(\mathcal{H})
        \ge d(\mathcal{H})
        \ge d(n-1,F), 
    \end{align*}
    which proves Fact~\ref{FACT:(F,P)-Turan-Pair}. 
\end{proof}

The proof for the following fact can be found in~{\cite[Lemma~4.2]{LMR1}} (with some minor modifications). 

\begin{fact}\label{FACT:small-deg-set-Z}
    Let $F$ be an $r$-graph and let $\mathcal{H}$ be an $n$-vertex $F$-free $r$-graph.  If $n$ is large, $\varepsilon >0$ is small, and $|\mathcal{H}| \ge \left(\pi(F) - \varepsilon\right)\binom{n}{r}$, then 
    \begin{enumerate}[label=(\alph*)]
    \item 
    the set 
    \begin{align*}
        Z_{\varepsilon}(\mathcal{H}) := \left\{v\in V(\mathcal{H}) \colon d_{\mathcal{H}}(v) \le \left(\pi(F)-2\varepsilon^{1/2}\right)\binom{n-1}{r-1}\right\}
    \end{align*}
    has size at most $\varepsilon^{1/2}n$, and 
    \item 
    the induced subgraph $\mathcal{H}'$ of $\mathcal{H}$ on $V(\mathcal{H})\setminus Z_{\varepsilon}(\mathcal{H})$ satisfies $\delta(\mathcal{H}') \ge \left(\pi(F)-3\varepsilon^{1/2}\right)\binom{n-1}{r-1}$. 
    \end{enumerate}    
\end{fact}

\subsection{Proofs of Theorems~\ref{THM:Turan-pair-smooth} and~\ref{THM:vertex-extend-bounded}}
We prove  Theorem~\ref{THM:Turan-pair-smooth} first. 

\begin{proof}[Proof of Theorem~\ref{THM:Turan-pair-smooth}]
    Fix an integer $n \ge 1$. Then 
    \begin{align*}
        |\delta(n,F) - d(n-1,F)|
        \overset{\text{Fact~\ref{FACT:(F,P)-Turan-Pair}}}&{=} \delta(n,F) - d(n-1,F) \\
        \overset{\text{Fact~\ref{FACT:min-degree-extremal}}}&{\le} d(n,F) - d(n-1,F)
        \overset{\text{Lemma~\ref{LEMMA:smooth-degree}}}{\le} 4\binom{n-1}{r-2},  
    \end{align*}
    which proves Theorem~\ref{THM:Turan-pair-smooth}. 
\end{proof}

Next we prove Theorem~\ref{THM:vertex-extend-bounded}. 

\begin{proof}[Proof of Theorem~\ref{THM:vertex-extend-bounded}]
    Fix constants $0 < \varepsilon \ll \varepsilon_1 \ll 1$ and let $n \in \mathbb{N}$ be sufficiently large. 
    Suppose to the contrary that there exists an $n$-vertex $F$-free $r$-graph $\mathcal{H}$ with $d(\mathcal{H}) \ge d(n,F) - \varepsilon\binom{n-1}{r-1}$ and $\Delta(\mathcal{H}) \ge d(n,F)+ \frac{1-\pi(F)}{8m}\binom{n-1}{r-1}$. Let $V = V(\mathcal{H})$. 
    Fix a vertex $v\in V$ with $d_{\mathcal{H}}(v) = \Delta(\mathcal{H})$. 
    Let $V_0 = V\setminus \{v\}$ and $\mathcal{H}_0 = \mathcal{H}[V_0]$. 
    Since 
    \begin{align*}
        |\mathcal{H}_0| 
        \ge |\mathcal{H}| -\binom{n-1}{r-1}
        \ge \mathrm{ex}(n,F) - 2\varepsilon\binom{n}{r}, 
    \end{align*}
    it follows from the edge-stability of $F$ that $\mathcal{H}_0$ contains a subgraph $\mathcal{H}_1$ with at least $\mathrm{ex}(n,F) - \varepsilon_1\binom{n}{r} \ge \left(\pi(F)-\varepsilon_1\right)\binom{n}{r}$ edges, 
    and moreover, $\mathcal{H}_1$ is a $P$-subconstruction. 
    
    It follows from Fact~\ref{FACT:small-deg-set-Z} that the set 
    \begin{align*}
        Z := \left\{v\in V \colon d_{\mathcal{H}_1}(v) \le \left(\pi(F)-2\varepsilon_1^{1/2}\right)\binom{n-1}{r-1}\right\}
    \end{align*}
    has size at most $\varepsilon_1^{1/2} n$, and moreover, the $r$-graph $\mathcal{H}_2 := \mathcal{H}_1[V_0\setminus Z]$ satisfies $\delta(\mathcal{H}_2) \ge \left(\pi(F)-3\varepsilon_1^{1/2}\right)\binom{n-1}{r-1}$. 
    Note that $\mathcal{H}_2 \subset \mathcal{H}_1$ is also a $P$-subconstruction. 

    Define $\mathcal{H}_3 := \mathcal{H}_2 \cup \left\{e \in \mathcal{H}[V\setminus Z] \colon v\in e\right\}$. 
    Since $|Z| \le \varepsilon_1^{1/2}n \le \frac{1-\pi(F)}{72m}\frac{n}{r}$, we have 
    \begin{align*}
        d_{\mathcal{H}_3}(v)
        \ge d_{\mathcal{H}}(v) - |Z|\binom{n-2}{r-2}
        & \ge d(n,F)+ \frac{1-\pi(F)}{8m}\binom{n-1}{r-1} - \frac{1-\pi(F)}{72m}\frac{n}{r}\binom{n-2}{r-2} \\
        & \ge d(n,F)+ \frac{1-\pi(F)}{8m}\binom{n-1}{r-1} - \frac{1-\pi(F)}{72m}\binom{n-1}{r-1} \\
        & \ge \left(\pi(F) + \frac{1-\pi(F)}{9m}\right) \binom{n-1}{r-1}. 
    \end{align*}
    Let $n' = |V\setminus Z|$. 
    Note that $\mathcal{H}_3$ is an $F$-free $r$-graph on $n'$ vertices with $\delta(\mathcal{H}_3) \ge \delta(\mathcal{H}_2) \ge \left(\pi(F)-3\varepsilon_1^{1/2}\right)\binom{n-1}{r-1}$, and 
    $v\in V(\mathcal{H}_3)$ is a vertex such that $\mathcal{H}_3 - v = \mathcal{H}_2$ is a $P$-subconstruction. 
    However, this contradicts the weak vertex-extendability of $F$ since $\varepsilon_1$ is sufficiently small and $d_{\mathcal{H}_3}(v) \ge \left(\pi(F) + \frac{1-\pi(F)}{9m}\right) \binom{n-1}{r-1}$. 
\end{proof}

\section{Proof of Theorem~\ref{THM:Hypergraph-Book-3}}\label{SEC:proof-3book}
The edge-stability of $\mathbb{F}_{3,2}$ was already proved in~{\cite[Theorem~2.2]{FPS053Book3page}}, 
so by Theorems~\ref{THM:Main-simple},~\ref{THM:Turan-pair-smooth}, and~\ref{THM:vertex-extend-bounded}, to prove Theorem~\ref{THM:Hypergraph-Book-3} it suffices to prove the following result. 

\begin{theorem}\label{THM:weak-vte-extend-F32}
    The $3$-graph $\mathbb{F}_{3,2}$ is weakly vertex-extendable with respect to the pattern $S_3 :=\left(2, \{1,2,2\}\right)$. 
\end{theorem}
%
%
\begin{proof}
    Fix $\delta >0$. Let $n$ be sufficiently large and $\zeta>0$ be sufficiently small. 
    Let $\mathcal{H}$ be an $n$-vertex $\mathbb{F}_{3,2}$-free $3$-graph with $\delta(\mathcal{H}) \ge \left(\frac{4}{9}-\zeta\right)\binom{n-1}{2}$. 
    Suppose that $v\in V$ is a vertex such that $\mathcal{H}_0 := \mathcal{H}-v$ is an $S_3$-subconstruction (i.e. semibipartite). 
    It suffices to show that $d_{\mathcal{H}}(v) \le \left(\frac{4}{9}+\delta\right)\binom{n-1}{2}$. 

    Suppose to the contrary that $d_{\mathcal{H}}(v) > \left(\frac{4}{9}+\delta\right)\binom{n-1}{2}$. 
    Let $V_1 \cup V_2$ be a bipartition of $V_0 := V\setminus \{v\}$ such that every edge in $\mathcal{H}_0$ contains exactly one vertex from $V_1$. 
    Since $|\mathcal{H}_0| \ge \frac{3}{n}\delta(\mathcal{H}) \ge \left(\frac{4}{9}-\zeta\right)\binom{n}{3}$, it follows from some simple calculations (see e.g.~{\cite[Theorem~2.2~(ii)]{FPS053Book3page}}) that 
    \begin{align}\label{equ:size-V1-V2-F32}
        \max\left\{\left||V_1|-\frac{n}{3}\right|, \left||V_2|-\frac{2n}{3}\right|\right\}
        \le \zeta^{1/2}n. 
    \end{align}
    Recall that the link of a vertex $u\in V(\mathcal{H})$ is defined as 
    \begin{align*}
        L_\mathcal{H}(u)
        :=\left\{A \in \binom{V(\mathcal{H})}{r-1}\colon A \cup \{u\}\in \mathcal{H}\right\}.
    \end{align*}
    Let $L = L_{\mathcal{H}}(v)$ for simplicity and let 
    \begin{align*}
        L_1:= L \cap \binom{V_1}{2}, \quad 
        L_2:= L \cap \binom{V_2}{2}, \quad\text{and}\quad
        L_{1,2}:= L \cap (V_1\times V_2). 
    \end{align*}
    Here we abuse the use of notation by letting $V_1\times V_2$ denote the edge set of the complete bipartite graph with parts $V_1$ and $V_2$.

\begin{figure}[htbp]
\centering
\subfigure
{
\begin{minipage}[t]{0.45\linewidth}
\centering
\tikzset{every picture/.style={line width=1pt}} 
\begin{tikzpicture}[x=0.75pt,y=0.75pt,yscale=-1,xscale=1,line join=round]

\draw   (383.91,183.99) .. controls (367.62,183.91) and (354.56,155.19) .. (354.74,119.84) .. controls (354.92,84.5) and (368.27,55.91) .. (384.56,55.99) .. controls (400.85,56.08) and (413.91,84.8) .. (413.73,120.14) .. controls (413.55,155.49) and (400.2,184.08) .. (383.91,183.99) -- cycle ;
Shape: Ellipse [id:dp8322898203147031] 
\draw   (495.87,235) .. controls (474.78,235.02) and (457.66,189.95) .. (457.63,134.35) .. controls (457.6,78.75) and (474.68,33.66) .. (495.77,33.65) .. controls (516.86,33.64) and (533.98,78.71) .. (534.01,134.31) .. controls (534.04,189.91) and (516.97,234.99) .. (495.87,235) -- cycle ;
\draw   [fill={rgb:red,1;green,0;blue,0}, fill opacity=0.5]  (503.35,133.8) .. controls (489,151) and (488.94,166.01) .. (501.13,178.23) .. controls (477,157) and (430.92,162.98) .. (384,148) .. controls (434,156) and (475,142) .. (503.35,133.8) --cycle;
\draw  [fill=uuuuuu, fill opacity=0.5]  (378.08,103.02) .. controls (397,137) and (400,179) .. (381,217) .. controls (442,149) and (469,141) ..  (503.35,133.8) .. controls (490.13,134.64) and (418,131) .. (378.08,103.02) -- cycle;

\draw  [fill=uuuuuu, fill opacity=0.5]  (378.08,103.02) .. controls (397,137) and (400,179) .. (381,217) .. controls (421,182) and (470,186) .. (501.13,178.23)  .. controls (470,182) and (413,163) .. (378.08,103.02) -- cycle;

\draw [fill=uuuuuu, fill opacity=0.5]   (378.08,103.02) .. controls (366,123) and (371.81,135.78) .. (384,148) .. controls (358,166) and (363,196) .. (381,217) .. controls (350,199) and (352.08,121.02) .. (378.08,103.02) -- cycle;

\draw [fill=uuuuuu] (378.08,103.02) circle (2pt);
\draw [fill=uuuuuu] (381,217) circle (2pt);
\draw [fill=uuuuuu]  (501.13,178.23) circle (2pt);
\draw [fill=uuuuuu] (503.35,133.8) circle (2pt);
\draw [fill=uuuuuu] (384,148)  circle (2pt);
\draw (375,225) node [anchor=north west][inner sep=0.75pt]   [align=left] {$v$};
\draw (378,155) node [anchor=north west][inner sep=0.75pt]   [align=left] {$u$};
\draw (375,87) node [anchor=north west][inner sep=0.75pt]   [align=left] {$w$};
\draw (503.13,181.23) node [anchor=north west][inner sep=0.75pt]   [align=left] {$a$};
\draw (507,126) node [anchor=north west][inner sep=0.75pt]   [align=left] {$b$};
\end{tikzpicture}
\end{minipage}
}
\subfigure
{
\begin{minipage}[t]{0.45\linewidth}
\centering
\tikzset{every picture/.style={line width=1pt}} 
\begin{tikzpicture}[x=0.75pt,y=0.75pt,yscale=-1,xscale=1,line join=round]

\draw   (156.91,176.99) .. controls (140.62,176.91) and (127.56,148.19) .. (127.74,112.84) .. controls (127.92,77.5) and (141.27,48.91) .. (157.56,48.99) .. controls (173.85,49.08) and (186.91,77.8) .. (186.73,113.14) .. controls (186.55,148.49) and (173.2,177.08) .. (156.91,176.99) -- cycle ;
\draw   (268.87,228) .. controls (247.78,228.02) and (230.66,182.95) .. (230.63,127.35) .. controls (230.6,71.75) and (247.68,26.66) .. (268.77,26.65) .. controls (289.86,26.64) and (306.98,71.71) .. (307.01,127.31) .. controls (307.04,182.91) and (289.97,227.99) .. (268.87,228) -- cycle ;
\draw  [fill=uuuuuu, fill opacity=0.5]  (276.35,126.8) .. controls (262,144) and (261.94,159.01) .. (274.13,171.23) .. controls (267.13,171.23) and (198,111) .. (151.08,96.02) .. controls (205,110) and (244,133) .. (276.35,126.8)  -- cycle;

\draw  [fill=uuuuuu, fill opacity=0.5]  (276.35,126.8) .. controls (262,144) and (261.94,159.01) .. (274.13,171.23)  .. controls (250,150) and (203.92,155.98) .. (157,141) .. controls (207,149) and (248,135) .. (276.35,126.8) --cycle;

\draw  [fill=uuuuuu, fill opacity=0.5]  (276.35,126.8) .. controls (262,144) and (261.94,159.01) .. (274.13,171.23)  .. controls (227,182) and (212,188)  .. (154,210) .. controls (205,176) and (248,159) .. (276.35,126.8) -- cycle;
\draw  [fill={rgb:red,1;green,0;blue,0}, fill opacity=0.5]  (151.08,96.02) .. controls (139,116) and (144.81,128.78) .. (157,141) .. controls (131,159) and (136,189) .. (154,210)  .. controls (123,192)  and (125.08,114.02).. (151.08,96.02);

\draw [fill=uuuuuu] (151.08,96.02) circle (2pt);
\draw [fill=uuuuuu] (157,141) circle (2pt);
\draw [fill=uuuuuu] (154,210) circle (2pt);
\draw [fill=uuuuuu] (276.35,126.8) circle (2pt);
\draw [fill=uuuuuu] (274.13,171.23) circle (2pt);
\draw (150,218) node [anchor=north west][inner sep=0.75pt]   [align=left] {$v$};
\draw (150,148) node [anchor=north west][inner sep=0.75pt]   [align=left] {$u$};
\draw (150,80) node [anchor=north west][inner sep=0.75pt]   [align=left] {$w$};
\draw (276.13,174.23) node [anchor=north west][inner sep=0.75pt]   [align=left] {$a$};
\draw (280,119) node [anchor=north west][inner sep=0.75pt]   [align=left] {$b$};
\end{tikzpicture}
\end{minipage}
}
\caption{Finding $\mathbb{F}_{3,2}$ in Claim~\ref{CLAIM:L2-large} (left) and Claim~\ref{CLAIM:L1-empty} (right).}
\label{fig:F32-weak}
\end{figure}
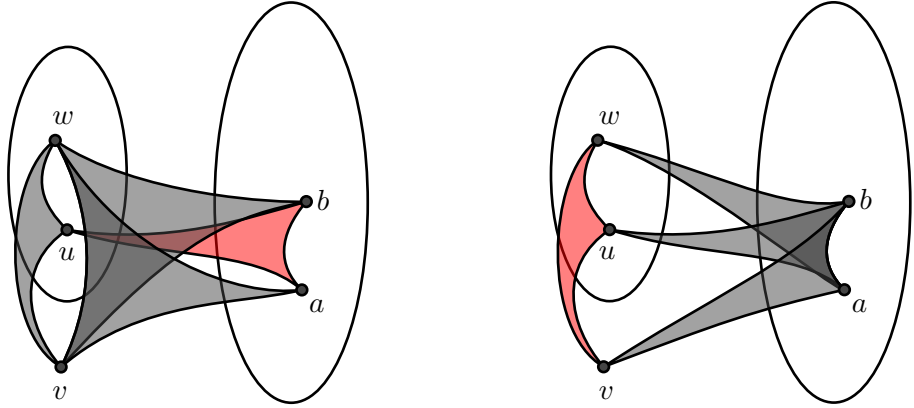
    \begin{claim}\label{CLAIM:L2-large}
        We have $|L_2| \ge \frac{\delta}{8}n^2$. 
    \end{claim}
    \begin{proof}
        Suppose to the contrary that $|L_{2}| \le \delta n^2/8$.
        Then it follows from the inequality
        \begin{align*}
            \sum_{v'\in V_1}d_{L}(v') 
            = 2|L_1|+|L_{1,2}|
            \ge |L| - |L_2|
            \ge \left(\frac{4}{9}+\delta\right)\binom{n-1}{2} - \frac{\delta}{8}n^2
            \ge \left(\frac{2}{9}+\frac{\delta}{4}\right)n^2
        \end{align*}
        that there exists a vertex $w\in V_1$ with 
        \begin{align*}
            d_{L}(w) 
            \ge \frac{\left(\frac{2}{9}+\frac{\delta}{4}\right)n^2}{\left(\frac{1}{3}+\zeta^{1/2}\right)n}
            \ge \left(\frac{2}{3}+\frac{\delta}{8}\right)n. 
        \end{align*}
        Therefore, by (\ref{equ:size-V1-V2-F32}),  we have 
        \begin{align*}
            \min\left\{|N_{L}(w) \cap V_1|, |N_{L}(w) \cap V_2|\right\} \ge \frac{\delta}{16}n. 
        \end{align*}
        Fix a vertex $u\in N_{L}(w) \cap V_1$ and let $V_2' = N_{L}(w) \cap V_2$. 
        Since 
        \begin{align}\label{equ:F32-001}
            \binom{|V_2|}{2} - d_{\mathcal{H}_0}(u) 
            \le \binom{\left(\frac{2}{3}+\zeta^{1/2}\right)n}{2} - \left(\frac{4}{9}-2\zeta\right)\binom{n-1}{2}
            < \binom{\delta n/16}{2}, 
        \end{align}
        there exists an edge $ab\in L_{\mathcal{H}}(u)\cap \binom{V_2'}{2}$. 
        However, this implies that $\mathbb{F}_{3,2} \subset \mathcal{H}[\{v,u,w,a,b\}]$ (see Figure~\ref{fig:F32-weak}), a contradiction.
    \end{proof}

    \begin{claim}\label{CLAIM:L1-empty}
        We have $L_1 = \emptyset$. 
    \end{claim}
    \begin{proof}
    Suppose to the contrary that there exists an edge $uw \in L_1$. 
    Note that $|L_{2}| \ge \delta n^2/8$ from Claim~\ref{CLAIM:L2-large}. 
    Choosing uniformly at random a pair $\{a,b\}$ from $\binom{V_2}{2}$, we obtain   
        \begin{align*}
            \min\left\{\mathbb{P}\left[ab\in L_{\mathcal{H}}(u)\right], \mathbb{P}\left[ab\in L_{\mathcal{H}}(w)\right] \right\}
            \ge \frac{\delta(\mathcal{H}_0)}{\binom{|V_2|}{2}}
             > \frac{\left(\frac{4}{9}-2\zeta\right)\binom{n-1}{2}}{\binom{\left(\frac{2}{3}+\zeta^{1/2}\right)n}{2}} 
             > 1- 10\zeta^{1/2}, 
        \end{align*}
        and 
        \begin{align*}
            \mathbb{P}\left[ab\in L_2\right]
            = \frac{|L_2|}{\binom{|V_2|}{2}}
            > \frac{\delta n^2/8}{\binom{\left(\frac{2}{3}+\zeta^{1/2}\right)n}{2}} 
            > \frac{\delta}{8}. 
        \end{align*}
        So it follows from the Union Bound that 
        \begin{align*}
            \mathbb{P}\left[ab\in L_2 \cap L_{\mathcal{H}}(u) \cap L_{\mathcal{H}}(w)\right]
            > 1 - \left(10\zeta^{1/2}+10\zeta^{1/2}+1-\frac{\delta}{8}\right)
            >0. 
        \end{align*}
        Hence, there exists an edge $ab\in L_2 \cap L_{\mathcal{H}}(u) \cap L_{\mathcal{H}}(w)$. 
        However, this implies that $\mathbb{F}_{3,2} \subset \mathcal{H}[\{v,u,w,a,b\}]$ (see Figure~\ref{fig:F32-weak}), a contradiction. 
    \end{proof}

\begin{figure}[htbp]
\centering
\tikzset{every picture/.style={line width=1pt}} 

\begin{tikzpicture}[x=0.75pt,y=0.75pt,yscale=-1,xscale=1,line join=round]

\draw   (248.91,169.99) .. controls (232.62,169.91) and (219.56,141.19) .. (219.74,105.84) .. controls (219.92,70.5) and (233.27,41.91) .. (249.56,41.99) .. controls (265.85,42.08) and (278.91,70.8) .. (278.73,106.14) .. controls (278.55,141.49) and (265.2,170.08) .. (248.91,169.99) -- cycle ;
\draw   (360.87,221) .. controls (339.78,221.02) and (322.66,175.95) .. (322.63,120.35) .. controls (322.6,64.75) and (339.68,19.66) .. (360.77,19.65) .. controls (381.86,19.64) and (398.98,64.71) .. (399.01,120.31) .. controls (399.04,175.91) and (381.97,220.99) .. (360.87,221) -- cycle ;
\draw  [fill=uuuuuu, fill opacity=0.5]   (239,132) .. controls (294,143) and (331,110) .. (354,88) .. controls (309,138) and (268,178)  .. (260,235) .. controls (261,184) and (271,154) .. (239,132) -- cycle;
\draw   [fill={rgb:red,1;green,0;blue,0}, fill opacity=0.5]  (379,162) .. controls (360,170) and (353,181) .. (352,199) .. controls (333,157) and (278.92,159.98)  .. (239,132) .. controls (278.92,159.98) and (361.65,167.2) .. (379,162) -- cycle;
\draw   [fill=uuuuuu, fill opacity=0.5]   (260,235) .. controls (307,184) and (330,149) .. (354,88) .. controls (347,121) and (331,166)  .. (352,199) .. controls (314,192) and (313,193) .. (260,235) -- cycle;

\draw   [fill=uuuuuu, fill opacity=0.5]   (260,235) .. controls (307,184) and (330,149) .. (354,88) .. controls (355,123) and (356,137) .. (379,162)  .. controls  (342,141) and (313,193) .. (260,235) -- cycle;
\draw [fill=uuuuuu] (379,162) circle (2pt);
\draw [fill=uuuuuu] (352,199) circle (2pt);
\draw [fill=uuuuuu]  (239,132) circle (2pt);
\draw [fill=uuuuuu] (260,235) circle (2pt);
\draw [fill=uuuuuu] (354,88)  circle (2pt);
\draw (255,242) node [anchor=north west][inner sep=0.75pt]   [align=left] {$v$};
\draw (230,114) node [anchor=north west][inner sep=0.75pt]   [align=left] {$u$};
\draw (347,70) node [anchor=north west][inner sep=0.75pt]   [align=left] {$w$};
\draw (359.13,195) node [anchor=north west][inner sep=0.75pt]   [align=left] {$a$};
\draw (377,142) node [anchor=north west][inner sep=0.75pt]   [align=left] {$b$};

\end{tikzpicture}

\caption{Finding $\mathbb{F}_{3,2}$ when $L_1 = \emptyset$.}
\label{fig:F32-weak-b}
\end{figure}
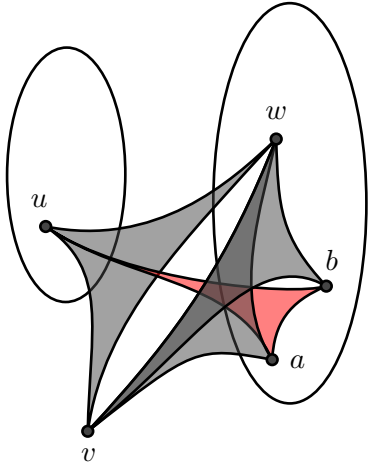

    Let us define 
    \begin{align*}
        U_1:=\left\{v'\in V_2 \colon |N_{L}(v') \cap V_1| \ge \frac{\delta}{16}n\right\}
        \quad\text{and}\quad
        U_2:=\left\{v'\in V_2 \colon |N_{L}(v') \cap V_2| \ge \frac{\delta}{16}n\right\}. 
    \end{align*}
    It follows from  
    \begin{align*}
    \left(\frac{1}{3}+\zeta^{1/2}\right)n |U_1|
        \ge \sum_{v'\in U_1}|N_{L}(v') \cap V_1|
        \ge |L_{1,2}| - \frac{\delta}{16}n |V_2\setminus U_1|
        \ge |L_{1,2}| - \frac{\delta}{16}n^2
    \end{align*}
    and 
    \begin{align*}
    \left(\frac{2}{3}+\zeta^{1/2}\right)n |U_2|
        \ge \sum_{v'\in U_2}|N_{L}(v') \cap V_2|
        \ge 2|L_{2}| - \frac{\delta}{16}n |V_2\setminus U_2|
        \ge 2|L_{2}| - \frac{\delta}{16}n^2  
    \end{align*}
    that 
    \begin{align*}
    |U_1|+|U_2|
    & \ge \frac{|L_{1,2}| - \frac{\delta}{16}n^2}{\left(\frac{1}{3}+\zeta^{1/2}\right)n} 
    + \frac{2|L_{2}| - \frac{\delta}{16}n^2}{\left(\frac{2}{3}+\zeta^{1/2}\right)n} \\
    & \ge \frac{|L_{1,2}| - \frac{\delta}{16}n^2 + |L_{2}| - \frac{\delta}{16}n^2}{\left(\frac{1}{3}+\zeta^{1/2}\right)n} \\
    & = \frac{|L| - \frac{\delta}{8}n^2}{\left(\frac{1}{3}+\zeta^{1/2}\right)n} 
    \ge \frac{\left(\frac{2}{9}+\frac{\delta}{4}\right)n^2 - \frac{\delta}{8}n^2}{\left(\frac{1}{3}+\zeta^{1/2}\right)n} 
     \ge \left(\frac{2}{3}+\frac{\delta}{8}\right)n. 
    \end{align*} 
    So it follows from (\ref{equ:size-V1-V2-F32}) that $|U_1 \cap U_2| \ge |U_1|+|U_2| - |V_2| \ge \frac{\delta}{16}n$. 

    Fix a vertex $w\in U_1 \cap U_2$ and a vertex $u\in N_{L}(w) \cap V_1$.
    Let $V_2' = N_{L}(w) \cap V_2$. 
    Since $|V_2'| \ge \frac{\delta}{16}n$, similar to (\ref{equ:F32-001}), there exists an edge $ab\in L_{\mathcal{H}}(u) \cap \binom{V_2'}{2}$. 
    However, this implies that $\mathbb{F}_{3,2} \subset \mathcal{H}[\{v,u,w,a,b\}]$ (see Figure~\ref{fig:F32-weak-b}), a contradiction.
    This completes the proof of Theorem~\ref{THM:weak-vte-extend-F32}. 
\end{proof}

\section{Concluding remarks}\label{SEC:Remarks}
By a small modification of the proof, 
one can easily extend Theorems~\ref{THM:Main-simple} and~\ref{THM:main-rainbow} to vertex-disjoint union of different hypergraphs as follows (here we omit the statement for the rainbow version). 

\begin{theorem}\label{THM:diff-disjoint-Turan}
    Let $m \ge r \ge 2, k \ge 1$ be integers and let $F_1, \ldots, F_{k}$ be nondegenerate $r$-graphs on at most $m$ vertices. 
    Suppose that there exists a constant $c > 0$ such that for all $i\in [k]$ and large $n \colon$
    \begin{enumerate}[label=(\alph*)]
    \item
        $F_i$ is $\left(c\binom{n}{r-1}, \frac{1-\pi(F)}{4m}\binom{n}{r-1}\right)$-bounded, and 
    \item
        $\mathrm{ex}(n, F_i)$ is $\frac{1-\pi(F)}{8m}\binom{n}{r-1}$-smooth. 
    \end{enumerate}
    Then there exist constant $N_0$ such that for all integers $n \ge N_0$ and $t_1, \ldots, t_k \in \mathbb{N}$ with $t+1:= \sum_{i=1}^{k}t_i \in [0, \varepsilon n]$, where $\varepsilon = \min\left\{\frac{c}{4erm}, \frac{1-\pi(F_1)}{64rm^2}, \ldots, \frac{1-\pi(F_k)}{64rm^2}\right\}$, we have 
    \begin{align*}
        \mathrm{ex}\left(n, \bigsqcup_{i=1}^{k}t_iF_i\right)
        \le \binom{n}{r}-\binom{n-t}{r} + \max_{i\in [k]}\left\{\mathrm{ex}(n-t, F_i)\right\}. 
    \end{align*}
    Moreover, if $\max_{i\in [k]} \mathrm{ex}(n-t, F_i)  = \mathrm{ex}(n, \{F_1, \ldots, F_k\})$, then the inequality above can be replace by equality. 
\end{theorem}

Recall that Allen, B\"{o}ttcher, Hladk\'{y}, and Piguet~\cite{ABHP15} determined, for large $n$, the value of $\mathrm{ex}(n, (t+1)K_3)$ for all $t\le n/3$. 
Considering that the situation is already very complicated for $K_3$, 
the following question seems very hard in general. 

\begin{problem}\label{PROB:Corradi-Hajnal-general}
Let $r \ge 2$ be an integer and $F$ be a nondegenerate $r$-graph with $m$ vertices. 
For large $n$ determine $\mathrm{ex}(n, (t+1)F)$ for all $t \le n/m$. 
\end{problem}

A first step towards a full understanding of Problem~\ref{PROB:Corradi-Hajnal-general} would be determining the regime of $t$ in which members in
$K_{t}^r \uproduct \mathrm{EX}(n-t, F)$ are extremal. 
Here we propose the following question, which seems feasible for many hypergraphs (including graphs). 

\begin{problem}\label{PROB:Corradi-Hajnal-general-phase-1}
Let $r \ge 2$ be an integer and $F$ be an $r$-graph with $m$ vertices. 
For large $n$ determine the maximum value of $s(n, F)$ such that 
\begin{align*}
    \mathrm{ex}(n, (t+1)F) 
    = \binom{n}{r}-\binom{n-t}{r} +\mathrm{ex}(n-t, F)
\end{align*}
holds for all $t \in [0, s(n,F)]$. 
\end{problem}

Understanding the asymptotic behavior of $s(n, F)$ would be also very interesting. 

\begin{problem}\label{PROB:Corradi-Hajnal-general-phase-1-asymp}
Let $r \ge 2$ be an integer and $F$ be an $r$-graph with $m$ vertices. 
Let $s(n, F)$ be the same as in Problem~\ref{PROB:Corradi-Hajnal-general-phase-1}. 
Determine the value of $\liminf_{n\to \infty} \frac{s(n,F)}{n}$. 
\end{problem}

Note that the result of Allen, B\"{o}ttcher, Hladk\'{y}, and Piguet~\cite{ABHP15} implies that $s(n, K_3) = \frac{2n-6}{9}$ for large $n$. In particular, $\lim_{n\to \infty} \frac{s(n,K_3)}{n} = \frac{2}{9}$. 

It would be also interesting to consider extensions of the density Corr\'{a}di--Hajnal Theorem to degenerate hypergraphs such as complete $r$-partite $r$-graphs and even cycles. 
The behavior for degenerate hypergraphs seems very different from nondegenerate hypergraphs, and we refer the reader to e.g.~{\cite[Theorem~1.3]{LMH23}} for related results on even cycles.  

\bibliographystyle{abbrv}
\bibliography{vertexdisjoint}
\end{document}